\documentclass[12pt]{amsart}

\usepackage{amsfonts,amsmath,latexsym,amssymb,verbatim,amsbsy,times}
\usepackage{amsthm}
\usepackage{graphicx}
\usepackage{caption}
\usepackage{subcaption}
\usepackage{dsfont}
\usepackage{tikz}
\usetikzlibrary{patterns}

\usepackage[margin=.8in]{geometry}
\usepackage[colorlinks=true, pdfstartview=FitV, linkcolor=blue,citecolor=green, urlcolor=blue]{hyperref}

\usepackage{enumitem}

\theoremstyle{plain}
\newtheorem{THEOREM}{Theorem}[section]

\newtheorem{LEMMA}[THEOREM]{Lemma}
\newtheorem{PROP}[THEOREM]{Proposition}

\theoremstyle{definition}
\newtheorem{DEF}[THEOREM]{Definition}

\theoremstyle{remark}
\newtheorem{REMARK}{Remark}[section]

\newcommand{\N}{\ensuremath{\mathbb{N}}}   
\newcommand{\R}{\ensuremath{\mathbb{R}}}   

\newcommand{\Q}{\ensuremath{\mathbb{Q}}}

\def \a {\alpha}
\def \b {\beta}
\def \d {\delta}
\def \g {\gamma}
\def \e {\varepsilon}

\def \i {\iota}
\def \n {\nabla}
\def \s {\sigma}
\def \th {\theta}

\def \W {\mathcal{W}}
\def \O {\Omega}

\def \bu {{\bf u}}

\def \id {{\bf id}}

\def \cE {\mathcal{E}}
\def\cF {\mathcal {F}}

\def \cM {\mathcal{M}}

\def \cP {\mathcal{P}}

\def \cR {\mathcal{R}}

\def \cW {\mathcal{W}}

\def \wtA {\widetilde{A}}
\def \tD {\widetilde{D}}
\def \tM {\widetilde{M}}
\def \trho {\widetilde{\rho}}
\def \M {\mathcal{M}}

\def\cprime{$'$}

\def \loc {\mathrm{loc}}

\def \< {\langle}
\def \> {\rangle}
\def \p {\partial}

\newcommand{\sgn}[1]{\mathrm{sgn}(#1)}

\DeclareMathOperator{\diam}{diam} %
\DeclareMathOperator{\supp}{supp} %
\DeclareMathOperator{\tr}{trace} %

\def \cP {\mathcal{P}}

\def \Lip {\mathrm{Lip}}

\def \dd  {\mathrm{d}}

\def \dt  {\, \mathrm{d}t}
\def \dt  {\, \mathrm{d}t}

\def \dx  {\, \mathrm{d}x}
\def \dy  {\, \mathrm{d}y}
\def \dz  {\, \mathrm{d}z}

\def \dr  {\, \mathrm{d}r}

\def \ds  {\, \mathrm{d}s}
\def \dv  {\, \mathrm{d}v}
\def \dw  {\, \mathrm{d}w}

\def \barx {\overline{x}}
\def \bary {\overline{y}}
\def \bart {\overline{t}}
\def \bars {\overline{s}}

\def\dbar{\,\dd\barx\,\dd\bart\,\dd\bary\,\dd\bars}

\title[Sticky particle Cucker--Smale dynamics and the Euler-alignment system]{Sticky particle Cucker--Smale dynamics and the entropic selection principle for the 1D Euler-alignment system}

\author{Trevor M. Leslie}
\address[T.M. Leslie]{Department of Mathematics, University of Southern California, Los Angeles, CA 90089}
\email{lesliet@usc.edu}

\author{Changhui Tan}
\address[C. Tan]{Department of Mathematics, University of South Carolina, Columbia, SC 29208}
\email{tan@math.sc.edu}



 \usepackage{setspace}

\begin{document}
	
\begin{abstract}
We develop a global wellposedness theory for weak solutions to the 1D Euler-alignment system with measure-valued density, bounded velocity, and locally integrable communication protocol. A satisfactory understanding of the low-regularity theory is an issue of pressing interest, as smooth solutions may lose regularity in finite time. However, no such theory currently exists except for a very special class of alignment interactions.  We show that the dynamics of the 1D Euler-alignment system can be effectively described by a nonlocal scalar balance law, the entropy conditions of which serves as an \emph{entropic selection principle} that determines a unique weak solution of the Euler-alignment system.  Moreover,  the distinguished weak solution of the system can be approximated by the \emph{sticky particle   Cucker--Smale dynamics}. Our approach is inspired by the work of Brenier and Grenier \cite{brenier1998sticky} on the pressureless Euler equations.  
\end{abstract}

\maketitle
\thispagestyle{empty} 

\vspace{-10 mm}

\setstretch{0.75}
\begin{small}
\tableofcontents
\end{small}

\setstretch{1}

\section{Introduction}

We are interested in the following Euler-alignment system
\begin{equation}
\label{e:EA}
\left\{ \begin{array}{rcl}
\p_t \rho + \n_x\cdot (\rho \bu) & = & 0\,, \qquad (x,t)\in \R^d\times \R_+, \\
\p_t (\rho \bu) + \n_x \cdot (\rho \bu \otimes \bu) & = &  \displaystyle \int_{\R^d} \rho(x,t)\rho(y,t) \phi(x-y) (\bu(y,t) - \bu(x,t))\dy,
\end{array}
\right.
\end{equation}
subject to the initial data
\[
\rho(x,0)=\rho^0(x),\quad \bu(x,0)=\bu^0(x).
\]
Here $\rho\ge 0$ and $\bu\in \R^d$ represent density and velocity, respectively.  We shall make the global assumption that $\rho$ is normalized to have total mass $1$. The term on the right-hand side of the equation for the momentum $\rho \bu$ is the
\emph{alignment} force. The function $\phi$ is called the \emph{communication protocol}, and it governs the strength of the interactions between the `agents' that comprise the density profile. Throughout the paper, we will assume $\phi$ is non-negative, locally integrable, and radially decreasing.

The Euler-alignment system comes from the theory of collective behavior.  Its salient feature is the nonlocal alignment interaction, which for appropriate $\phi$ leads to a remarkable long-time behavior referred to as \emph{flocking} (a term intentionally reminiscent of a group of birds).  We will discuss flocking in due course, but we do not attempt a comprehensive overview here. The most complete reference is \cite{ShvydkoyBook}, which also contains references to many other excellent reviews. 

The nonlocality of \eqref{e:EA} is a notorious difficulty in the study of the wellposedness and long-time behavior for this equation.  However, there are two important cases where the nonlocality drops out.  If $\phi \equiv 0$,  then \eqref{e:EA} reduces to the well-studied \emph{pressureless Euler equations}. This case does not exhibit the alignment features associated with non-degenerate $\phi$, but its more developed wellposedness theory showcases an arsenal of tools that one can try out on the Euler-alignment system.  The other situation where nonlocality is not truly present is that of \emph{all-to-all coupling}, where $\phi$ is a positive constant.  Solutions of the all-to-all coupled system exhibit most features of the long-time behavior that one expects for more general $\phi$. However, the analysis of this  case is simpler: the alignment force reduces to a linear and local damping.  

In this paper, we develop a global wellposedness theory for weak solutions of the 1D Euler-alignment system \eqref{e:EA} with measure-valued density, and bounded velocity.  Our analysis covers the classical setup when the communication protocol $\phi$ is bounded and Lipschitz. More interestingly, it also works for the case when $\phi$ is \emph{weakly singular}, namely it has an integrable singularity at the origin. 
We show an asymptotic flocking behavior for the solutions we construct.  Our approach adapts the \emph{sticky particle} approximation, originally developed by Brenier and Grenier \cite{brenier1998sticky} to treat the 1D pressureless Euler equations. We require a detailed understanding of the relationship between the discrete and hydrodynamic settings; let us therefore review the derivation of \eqref{e:EA} from the Cucker--Smale system. 

\subsection{A brief derivation of the Euler-alignment system}

\label{ss:Derivation}

The Euler-alignment system can be derived as a hydrodynamic version of the celebrated Cucker--Smale system of ODE's \cite{CS2007a,CS2007b}:
\begin{equation}
\label{e:CS}
\begin{cases} 
\displaystyle\frac{\dd x_i}{\dd t} = v_i \\
\displaystyle\frac{\dd v_i}{\dd t} =  \sum_{\substack{j=1\\x_j\ne x_i}}^N m_j \phi(x_j - x_i) (v_j - v_i),
\end{cases}
\qquad 
(x_i,v_i)\in \R^d\times \R^d,\quad i=1,\cdots,N.
\end{equation}
This system governs the motion of $N$ agents with  masses $m_i\ge 0$, positions $x_i$ and velocities $v_i$. As the number of agents goes to infinity, one can derive a kinetic formulation, using BBGKY hierarchies \cite{HT2008} or mean-field limits \cite{HaLiu2008, carrillo2010particle, MuchaPeszek2018}.  The kinetic distribution function $f(x,v,t)$ solves the Vlasov-type equation 
\begin{equation}
\label{e:kineticCS}
\renewcommand*{\arraystretch}{1.5}
\left\{ \begin{array}{l}
\p_t f + v\cdot \n_x f + \n_v \cdot \big(fF(f)\big) = 0, \\
F(f)(x,v,t) = \displaystyle\int_{\R^d\times\R^d} f(y,w,t)\phi(x-y)(w-v)\dw\dy,
\end{array}\right.
\qquad (x,v,t)\in \R^d \times \R^d \times \R_+.
\end{equation}
Define the macroscopic density and momentum by
\[
\rho(x,t) = \int_{\R^d} f(x,v,t)\dv, 
\qquad 
\mathbf{P}(x,t)=\rho(x,t) \bu(x,t) = \int_{\R^d} vf(x,v,t)\dv.
\]
Here $\bu(\cdot, t)$ is the macroscopic velocity,  well-defined on $\O(t):=\{x:\rho(x,t)>0\}$.  Taking zeroth and first moments on $v$ of the kinetic system then yields 
\begin{equation}
\label{e:EAReynolds}
\renewcommand*{\arraystretch}{1.8}
\left\{ \begin{array}{rcl}
\p_t \rho + \n_x\cdot (\rho \bu) & = & 0, \qquad \qquad (x,t)\in \R^d\times \R_+,\\
\empty \p_t (\rho \bu) + \n_x \cdot (\rho \bu \otimes \bu+\cR) & = & 
\displaystyle \int_{\R^d} \rho(x,t)\rho(y,t) \phi(x-y) (\bu(y,t) - \bu(x,t))\dy,\\
\cR(x,t) & = & \displaystyle\int_{\R^d} (v-\bu(x,t))\otimes (v-\bu(x,t))f(x,v,t)\dv,
\end{array}
\right.
\end{equation}
where $\cR$ denotes the Reynolds stress tensor.  Finally, one obtains the pressureless Euler-alignment system \eqref{e:EA} as a hydrodynamic limit by taking a \emph{monokinetic   ansatz} $f(x,v,t) = \rho(x,t)\d(v-\bu(x,t))$, which eliminates $\cR$ from \eqref{e:EAReynolds}. Rigorous justification (for bounded $\phi$) can be found in \cite{KangVasseur2015,FigalliKang2019,ShvydkoyBook}.

An alternative hydrodynamic limit can be obtained \cite{KMT2013,KMT2014,KMT2015} by taking an \emph{ isothermal ansatz} $f(x,v,t) = (2\pi)^{-d/2}\rho(x,t)\exp(-|v-\bu(x,t)|^2/2)$, in which case $\n_x\cdot\cR$ becomes a pressure term $\n_x \rho$. The resulting \emph{isothermal Euler-alignment system} was investigated in \cite{CCTT2016,ChoiEAPressure2019}. A more general class of \emph{isentropic Euler-alignment system} with a pressure term $\n_x(\rho^\gamma),  \gamma\geq1$ was considered in \cite{ConstantinDrivasShvydkoy2020,chen2020global}. The regularity theories of these cases are less well-developed than that of their pressureless cousin.  We do not treat the pressured system further in the present work, as the assumption of monokineticity is strongly embedded in our framework.
  
Finally, many authors prefer for technical reasons to replace \eqref{e:EA}$_2$ with the velocity equation
\begin{equation}
\label{e:EAv}
\p_t \bu + \bu \cdot \n \bu = \int \rho(y,t) \phi(x-y)(\bu(y,t) - \bu(x,t))\dy,
\end{equation}
and to define $\bu(\cdot, t)$ on all of $\R^d$, or on some domain containing $\supp\rho(t)$.  We will not use the equation \eqref{e:EAv} directly; the divergence form of \eqref{e:EA} is more amenable to the framework we want to build.  We mostly do not distinguish between the two formulations in our discussion of the literature.

\subsection{Existing wellposedness theory on the pressureless Euler-alignment system}
\label{ss:WPoverview}


The global regularity theory for smooth solutions of the Euler-alignment system \eqref{e:EA} is fairly well-established in one space dimension, with different types of alignment interactions.

The first scenario is when the interaction is \emph{regular}, namely $\phi$ is bounded and Lipschitz.
At the heart of many wellposedness results is the following quantity, introduced in~\cite{CCTT2016}:
\begin{equation}\label{eq:e}
e(x,t) = \p_x u(x,t) + \phi*\rho(x,t).
\end{equation}
Here $*$ denotes convolution in the spatial variable. Remarkably, $e$ satisfies the simple evolution equation
\begin{equation}\label{eq:e}
  \p_t e + \p_x (ue) = 0.
\end{equation}
This structure yields a precise description of global wellposedness via a \emph{critical threshold} condition \cite{CCTT2016}: if the initial condition is \emph{subcritical}, satisfying $e^0\geq0$ on all of $\R$, then the solution stays globally regular; otherwise, \emph{supercritical} initial data lead to shock formations in finite time.  

The next scenario is when the interaction is \emph{weakly singular}, that is, when $\phi$ has an integrable singularity at the origin, e.g. $\phi(r)\sim r^{-s}$ with $s\in(0,1)$. Under this setup, a similar but distinct critical threshold condition holds \cite{Tan2020euler}; the main additional subtlety occurs when $e^0$ takes the value $0$.

In the presence of vacuum $\rho^0(x)=0$, the critical threshold conditions are inaccessible from physical initial data $(\rho^0,P^0)$.  The first author \cite{L2019CTC} has  worked with an antiderivative $\psi$ of $e$, defined by
\begin{equation}
\label{e:psidef}
\psi(x,t) = u(x,t) +  \Phi*\rho(x,t),
\end{equation}
where $\Phi$ is the unique odd antiderivative of $\phi$.
The critical threshold condition can then be expressed in terms of $\psi^0$, which must be nondecreasing \emph{on the support of $\rho^0$} to propagate regularity. This condition is also sufficient when the protocol $\phi$ is regular; for weakly singular interactions, additional assumptions are needed to propagate regularity. Sharp conditions are not known. In the present work, we make extensive use of the quantity $\psi$ and its discrete analog for the Cucker--Smale system.

Another interesting scenario is when the interaction is \emph{strongly singular}, namely $\phi$ has a \textit{non}-integrable singularity at the origin. In this case, the alignment produces nonlocal dissipation that regularizes the solution. It has been shown that solutions are globally regular for \emph{all} smooth initial data away from vacuum \cite{DKRT,KT,ShvydkoyTadmorI,ShvydkoyTadmorII,ShvydkoyTadmorIII}. (However, the non-vacuum stipulation is important here, c.f. \cite{Tan2019singularity, ArnaizCastro2019}.) Moreover, for rough initial data, the solutions are instantaneously regularized thanks to the strongly singular interaction, see e.g. \cite{DanchinMuchaPeszekWroblewski2018, Leslie2019, Lear2021unising, bai2022global}.

The Euler-alignment system \eqref{e:EA} is less  well understood in higher dimensions, largely because of the lack of a scalar quantity $e$ that solves a simple continuity equation like \eqref{eq:e}. A natural candidate for a multi-dimensional replacement is $e = \n_x \cdot \bu + \phi*\rho$, which satisfies the equation $\p_t e+\n_x \cdot (\bu e) = (\n_x\cdot \bu)^2 - \tr((\n_x \bu)^2)$.  The right-hand side vanishes if the velocity is \emph{unidirectional}, i.e., $\bu(x,t) = u(x,t)\mathbf{h}$ with a fixed direction $\mathbf{h}\in\R^d$; in this case, the same threshold as in the 1D setting holds \cite{LearShvydkoy2019}.  However, for general $\bu$, the term $(\n_x\cdot \bu)^2 - \tr((\n_x \bu)^2)$ does not vanish and is difficult to control (but c.f. \cite{LearShvydkoy2019} for the `almost unidirectional' case).  Partial results are available in 2D \cite{TT2014,HT2016}, for radial solutions \cite{tan2021eulerian}, and recently in higher dimensions \cite{Tadmor2021review}.

As for the asymptotic behavior,
Ha and Liu \cite{HaLiu2008} have shown that the Cucker--Smale dynamics
\eqref{e:CS} enjoy the \emph{flocking} property: when the
communication weight $\phi$ has a \emph{fat tail}, namely
\begin{equation}\label{eq:fattail}
  \int_1^\infty\phi(r)\dr=\infty,
\end{equation}
then the diameter of the positions of all agents remains
uniformly bounded in time, and moreover, the velocities of all agents
tend to a common value as time goes to infinity.  An analog of this
property is inherited by the Euler-alignment system \eqref{e:EA}, at
least for smooth solutions: it was proved in \cite{TT2014} that if
$\phi$ satisfies \eqref{eq:fattail}, then \emph{strong solutions must
flock}.

\subsection{Weak solutions and the non-uniqueness issue}
Though the theory of strong solutions to the Euler-alignment system has been well-developed over the last decade, little is known about weak solutions, even in one dimension.  This is a serious gap in the theory, as solutions can lose regularity even if they are initially smooth. Several natural questions arise when considering the possibility of weak solutions:
\begin{itemize}
  \item How does a solution evolve after the formation of a shock?
  \item Do weak solutions flock?
  \item How are weak solutions connected to the Cucker--Smale dynamics?
\end{itemize}
The global wellposedness theory we develop in this paper will address these points.

Let us remark that when the interaction is strongly singular and the initial data non-vacuous, due to instant regularization, solutions are smooth and the existing theory of strong solutions applies. Therefore, we shall focus on the case when the communication protocol $\phi$ is \emph{locally integrable}, so the interaction can be either regular or weakly singular. A theory of weak solutions is needed, especially for supercritical initial data.

It is not difficult to formulate a satisfactory definition of a
distributional weak solution of  the Euler-alignment system
\eqref{e:EA}. However, it is well-known that such solutions are not
unique. In fact, Carrillo et al. \cite{Carrilloetal2017weakEA} studied weak solutions
to Euler systems with a general class of nonlocal interactions and
showed that there exist infinitely many weak solutions that dissipate
the kinetic energy. What we need, therefore, is an additional \emph{selection
  principle} that will single out a unique weak solution.

The non-uniqueness issue is better understood for the 1D
pressureless Euler system
\begin{equation}
\label{e:PE}
\left\{ \begin{array}{rcl}
\p_t \rho + \p_x(\rho u) & = & 0,  \\
\p_t (\rho u) + \p_x(\rho u^2) & = &  0.
\end{array}
\right.
\end{equation}

A well-known class of entropy inequalities for \eqref{eq:PE}, introduced by Lax \cite{Lax1971Entropy}, requires that for any positive convex entropy $\eta$,
\begin{equation}\label{e:entropic}
  \p_t(\rho \eta(u)) + \p_x(\rho u \eta(u))\le 0.
\end{equation}
These entropy inequalities can be adapted to the 1D Euler-alignment system, by replacing $\eta(u)$ with $\eta(\psi)$, where $\psi$ is defined in \eqref{e:psidef}. However, the entropy inequalities \eqref{e:entropic} do not guarantee uniqueness. Bouchut \cite{Bouchut1994} showed there are infinitely many entropic solutions to the 1D pressureless Euler equations that satisfy \eqref{e:entropic}. He also presented an instructive example with atomic initial data: Consider a configuration where two
particles move towards each other with velocities $v_1>v_2$. One can impose different rules when they collide. As long as the collision preserves momentum, and the post-collision velocities $v_1', v_2' \in [v_2, v_1]$, this setup generates an entropic solution
for the 1D pressureless Euler equations. Two particular solutions are: (i) no collision: $v_i'=v_i, i=1,2$; (ii) completely inelastic collision: $v_1'=v_2'$. Hence, a stronger selection principle is required to obtain a unique solution.

\subsection{Sticky particle dynamics and selection principles}

Among all the collision rules, the completely inelastic collision dissipates the most energy. Since the post-collision velocities are the same, the two particles stick to each other and travel together after the collision.  This \emph{sticky particle} model was originally proposed by Zeldovich \cite{Zeldovich1969sb}; it generates atomic weak solutions to \eqref{e:PE}.  Grenier \cite{Grenier1995} also used the sticky particle dynamics to prove existence (but not uniqueness) of solutions to \eqref{e:PE}.  The sticky particle dynamics underlie at some level all of the successful selection principles that we discuss below for the theory of \eqref{e:PE} and related systems.  

One line of results on the theory of \eqref{e:PE} is related to a so-called \textit{generalized variational principle} due to E, Rykov, and Sinai \cite{ERykovSinai1996}, which is compatible with the sticky particle dynamics and can serve as a selection principle.  Huang and Wang \cite{HuangWang2001stickyunique} proved existence and uniqueness of weak solutions satisfying the \textit{one-sided Lipschitz condition}
\begin{equation}\label{eq:onesideLip}
\frac{u(x_2,t)-u(x_1,t)}{x_2-x_1}\le \frac{1}{t},\quad t>0. 
\end{equation}
Their construction is based on a \textit{generalized potential} and has strong ties to the variational principle of \cite{ERykovSinai1996} (c.f. also the earlier works \cite{WangDing1997,WangHuangDing1997}). The framework of \cite{HuangWang2001stickyunique} has been adapted to the 1D Euler-alignment system with all-to-all coupling \cite{HaHuangWang2014wk, Jin2016flockingdiss}, where the alignment interaction reduces to a linear, local damping. It is not clear whether this approach can be extended to the truly nonlocal case of a general communication protocol $\phi$; to our knowledge, this question has not been treated in the literature.  

Another type of selection principle for \eqref{e:PE} is  based on the entropy conditions for a related scalar equation.  This approach was pioneered by Brenier and Grenier \cite{brenier1998sticky}, who study \eqref{e:PE} by connecting it to the \emph{scalar conservation law}
\begin{equation}\label{eq:PE}
\p_t M + \p_x(A(M)) = 0,
\end{equation}
where $M$ is the cumulative distribution function of the density $\rho$, and the flux $A$ only depends on the initial data $\rho^0$ and $u^0$. The entropy conditions for \eqref{eq:PE} select a unique solution $M$, which determines a distinguished weak solution of the pressureless Euler equations through $\rho = \p_x M$, $\rho u = \p_x(A(M))$. Hence, the entropy inequalities for the scalar conservation law \eqref{eq:PE} serve as a satisfactory selection principle for \eqref{e:PE}. Moreover, the entropy solution of \eqref{eq:PE} has an elegant connection to the sticky particle dynamics:  with a discretized initial condition $M_N^0$ and flux $A_N$, the sticky particle dynamics generate the entropy solution of the scalar conservation law by tracking the locations of all the shocks of $M_N$.

Bouchut and James have developed a related but alternative theory of \textit{duality solutions} \cite{BouchutJames1998, BouchutJames1999} for solutions to \eqref{e:PE}.  Their theory relies on properties of monotone solutions to \eqref{eq:PE}, and they prove uniqueness under an assumption similar to (but stronger than) \eqref{eq:onesideLip}.  The framework of \cite{brenier1998sticky} has been successfully applied to the 1D Euler--Poisson equations \cite{BrenierGangboSavareWestdickenberg2013,nguyen2008pressureless,NguyenTudorascu2015}.  To handle the additional nonlocal Poisson force, a nonlocal scalar conservation law \eqref{eq:PE} is generated, with a time dependent flux $A=A(M,t)$. A similar argument also works for the 1D Euler-alignment system with all-to-all coupling \cite{Jin2015flockingdiss}.

In addition to the literature above, we also want to mention a very recent work of Amadori and Christoforou \cite{amadori2021bv} on the global existence and asymptotic behavior of weak solutions to the Euler-alignment system with all-to-all coupling and an additional pressure term.  Their approach is generated by front-tracking approximation but otherwise has little overlap with the works described above.

The sticky particle approach to the pressureless Euler equations and related systems has continued to garner attention, through the lens of optimal transport \cite{NatileSavare2009,BrenierGangboSavareWestdickenberg2013,CavallettiSedjroWestdickenberg2015, nguyen2008pressureless, NguyenTudorascu2015}, and also from a probabilistic perspective \cite{Dermoune1999, DermouneMoutsinga2003, Moutsinga2008, Hynd2019LagrangianSP, Hynd2020trajectory, Hynd2020probmeasureSP}.  An exhaustive review of the literature on the pressureless Euler system is far beyond the scope of this paper.  However, the approaches described above provide sufficient context for the ideas developed in our paper.

\subsection{Scalar balance laws for the Euler-alignment system}
We are interested in the 1D Euler-alignment system with general alignment interactions. 
\begin{equation}
\label{eq:1DEA}
\left\{
\begin{array}{rcl}
\p_t \rho + \p_x (\rho u) & = & 0, \\
\p_t (\rho u) + \p_x (\rho u^2) & = & \rho (\phi*(\rho u)) - \rho u (\phi*\rho).
\end{array}
\right.
\end{equation}
The major challenge is to appropriately treat the nonlocality from the alignment interactions. Unfortunately, the system cannot be connected to a scalar conservation law of the type \eqref{eq:PE}, except for special cases, e.g. constant $\phi$.

We introduce a new scalar \emph{balance} law connected to the 1D Euler-alignment system:
\begin{equation}
  \label{e:balancelaw}
  \p_t M + \p_x(A(M)) = (\phi*M) \p_x M.
\end{equation}
The contribution from the alignment interaction is split between the flux term and the nonlinear, nonlocal right-hand side of \eqref{e:balancelaw}. Following the framework of Brenier and Grenier, we establish a global wellposedness theory for \eqref{e:balancelaw} and show that there is a unique entropy solution. We then construct a unique weak solution to the 1D Euler-alignment system \eqref{eq:1DEA}, using the entropy conditions for \eqref{e:balancelaw} as our selection principle. The intrinsic nonlocality embedded in \eqref{e:balancelaw} requires a significant advancement of the analytical techniques used for \eqref{eq:PE}, in terms of both the global wellposedness theory and (more significantly) the precise connection with \eqref{eq:1DEA}.  For example, we prove that solutions $M$ of \eqref{e:balancelaw} are stable with respect to perturbations of the initial data $M^0$ and the flux $A$.  Our bound is the same as the one for scalar conservation laws (c.f. \cite[Theorem 3]{Lucier1986}); however, we need to use monotonicity of $M(\cdot, t)$ (built into our definition of entropy solution) in an essential way to treat the nonlocal term, whereas this assumption is not needed to treat \eqref{eq:PE}.  Another distinctive feature of \eqref{e:balancelaw} is that $\p_x(A(M))$ does \emph{not} represent the momentum $\rho u$; rather, it is equal to $\rho \psi$. We use the relation \eqref{e:psidef} to recover the momentum. Finally, when $\phi$ is merely locally integrable, the fact that $\phi*M$ may not be differentiable presents challenges in both the existence and stability proofs: The term corresponding to $(\phi*M)\p_x M$ in our definition of entropy solution is the most subtle with respect to convergence of the discretization.  Furthermore, it necessitates an additional application of the BV chain rule and the use of a delicate cancellation in our uniqueness and stability argument.

We also explore the connection between \eqref{e:balancelaw} and the \emph{sticky particle Cucker--Smale} dynamics, introduced in Section \ref{s:Sticky}.  We show that the entropy solution of \eqref{e:balancelaw} can be constructed through an approximation by a sequence of solutions to the sticky particle Cucker--Smale dynamics. Ultimately, our approach yields a uniquely determined solution to \eqref{eq:1DEA}, through a discrete approximation by sticky particle Cucker--Smale dynamics, for measure-valued data.  Under additional assumptions on the initial conditions and protocol $\phi$, we obtain an error estimate for the approximation, with an explicit convergence rate of up to $O(N^{-1})$. The rate echoes the
work of Lucier \cite{Lucier1986} on scalar conservation laws. 
This paves a clean path for the numerical implementation of our
solution through the sticky particle approximation.

Finally, we study the asymptotic behavior of our constructed solution to \eqref{eq:1DEA}. Applying uniform estimates on the sticky particle Cucker--Smale dynamics and the convergence result of the approximation,  we establish the same flocking property that is enjoyed by strong solutions: If $\phi$ satisfies \eqref{eq:fattail}, our \emph{weak solutions must flock}.


\subsection{Main results and structure of the paper}
We study the following three systems and their connections. The sticky particle collision rule in the Cucker--Smale dynamics corresponds to the entropy conditions for the scalar balance law, which in turn serves as the selection principle to the unique weak solution of the 1D Euler-alignment system.

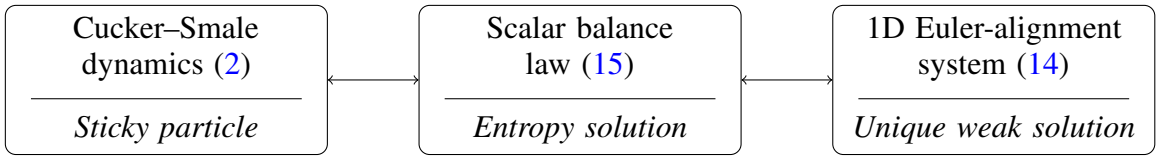
\begin{center}
\tikzstyle{block} = [rectangle, draw, text width=4cm, text centered, rounded corners, minimum height=3em]
\begin{tikzpicture}
  \node [block] (CS) at (0,0) {Cucker--Smale dynamics \eqref{e:CS}\\ \bigskip
    \textit{Sticky particle}};
  \node [block, right of=CS, node distance=5.5cm] (BL) {Scalar balance\\
    law \eqref{e:balancelaw} \\ \bigskip \textit{Entropy solution}};
  \node [block, right of=BL, node distance=5.5cm] (EA) {1D
    Euler-alignment system \eqref{eq:1DEA}\\ \bigskip \textit{Unique weak
      solution}};
  \draw (-1.8,-.25) -- (1.8,-.25);
  \draw (3.7,-.25) -- (7.3,-.25);
  \draw (9.2,-.25) -- (12.8,-.25);
  \path[draw, <->] (CS) --  (BL);
  \path[draw, <->] (BL) --  (EA);
\end{tikzpicture}
\end{center}

Our results are summarized in the following points.
The explicit statements appear in the text.


\begin{itemize}
	\item \textbf{Existence, Uniqueness, and Stability.}
	\begin{itemize}[leftmargin = 5 mm]
		\item For monotone initial data $M^0$ and Lipschitz flux $A$, the scalar balance law \eqref{e:balancelaw} has a unique entropy solution $M$, which is $BV$ in space and time (Theorem~\ref{t:M}(a)).  The solution is stable under perturbations of both $M^0$ and $A$ (Theorem~\ref{t:M}(c)).
		\item Given $\rho^0\in \cP_c(\R)$ (a compactly supported probability measure) and $u^0\in L^\infty(\dd\rho^0)$, a unique solution of \eqref{eq:1DEA} can be generated from the unique entropy solution of \eqref{e:balancelaw} with corresponding initial data $M^0$ and flux $A$ (Theorem~\ref{def:solution}, the \textit{entropic selection principle}).
	\end{itemize}
	\item \textbf{Discretization and Approximability.}
	\begin{itemize}[leftmargin = 5 mm]
		\item The sticky particle Cucker--Smale dynamics    determine the entropy solution of the scalar balance law \eqref{e:balancelaw} for discretized initial data $M_N^0$ and flux $A_N$ (Theorem \ref{t:discr}).
		\item The sticky particle Cucker--Smale dynamics approximate the unique solution of \eqref{e:balancelaw} and \eqref{eq:1DEA} for general initial data (Theorems \ref{t:M}(b) and \ref{t:SPconv}), with an explicit convergence rate depending on $\phi$ and $u^0$ (Theorem~\ref{t:approxrate}).
	\end{itemize}
	\item \textbf{Long-Time Behavior and Flocking.}
	\begin{itemize}[leftmargin = 5 mm]
		\item Our constructed weak solution to the 1D Euler-alignment system \eqref{eq:1DEA} exhibits the flocking phenomenon (Theorem~\ref{t:EAflocking}): If $\phi$ which decays slowly enough, the velocity $u$ converges to a constant, and $\rho$ converges to a traveling wave $\rho_\infty$.
	\end{itemize}	
\end{itemize}

\subsubsection*{Outline of the paper}
In Section \ref{s:preliminaries}, we give a formal derivation of the scalar balance law \eqref{e:balancelaw} from the Euler-alignment system and derive the associated Rankine--Hugoniot condition and Oleinik entropy condition.  In Section \ref{s:Sticky}, we discuss the properties of the sticky particle Cucker--Smale dynamics that are needed for our wellposedness theory. 
In Section \ref{s:discbalancelaw}, we connect the sticky particle Cucker--Smale dynamics to the entropy solution of the discretized scalar balance law \eqref{e:balancelaw}. In Section \ref{s:M}, we present the wellposedness theory for entropy solutions of \eqref{e:balancelaw} and prove the convergence of the sticky particle approximation. In Section \ref{s:EArecovery}, we establish rigorously the connection between solutions to the scalar balance law \eqref{e:balancelaw} and the 1D Euler-alignment system \eqref{eq:1DEA}, construct a unique weak solution, and study the approximation by the sticky particle Cucker--Smale dynamics. Finally, in Section \ref{s:MoreProps}, we study the asymptotic flocking behavior. 

\section{Derivation of the scalar balance law and entropy conditions}\label{s:preliminaries}

\subsection{The scalar balance law}
\label{ss:formalderiv}

We give a formal derivation of the scalar balance law \eqref{e:balancelaw} from the 1D Euler-alignment system \eqref{eq:1DEA}, assuming all functions involved are as regular as necessary. Rigorous justification of the equivalence between the two systems will be made in Section \ref{s:EArecovery}.

We begin by reformulating the 1D Euler-alignment system as follows. Integrating \eqref{eq:e} yields
\begin{equation}\label{eq:psi}
  \p_t\psi+u\p_x\psi=0,
\end{equation}
where $\psi = u + \Phi*\rho$ as in \eqref{e:psidef}, and 
\begin{equation}
\label{e:Phidef}
\Phi(x) = \int_0^x \phi(y)\dy.
\end{equation}
The 1D Euler-alignment system can then be expressed in terms of the pair $(\rho, \rho \psi)$:
\begin{equation}
\label{e:rhorhopsi}
\left\{ \begin{array}{rcl}
\p_t \rho + \p_x (\rho u) & = & 0, \\
\p_t (\rho \psi) + \p_x (\rho \psi u) & = & 0.
\end{array} \right.  
\end{equation}
The velocity can be recovered from $\psi$ and $\rho$ via the relation \eqref{e:psidef}:
\begin{equation}
\label{e:urhorhopsi}
u = \psi - \Phi * \rho.
\end{equation}
Since $\rho$ and $\rho \psi$ satisfy the same continuity equation
\eqref{e:rhorhopsi}, their primitives
\begin{equation}\label{eq:MQ}
M(x,t) = - \frac12 + \int_{(-\infty,x]} \rho(y,t)\dy,
\qquad Q(x,t) = \int_{(-\infty,x]} \rho \psi(y,t)\dy
\end{equation}
will satisfy the transport equations
\[
\p_t M + u\p_x M  =  0, \qquad
\p_t Q + u\p_x Q  =  0.
\]

Let $A$ be a map (depending only on the initial data $(M^0, Q^0)$) such that $A\circ M^0 = Q^0$. If $(x,t)$ lies on a characteristic path originating at $(x^0,0)$ and governed by the velocity field $u$, then 
\[
Q(x,t) = Q^0(x^0)=A(M^0(x^0))=A(M(x,t)).
\]
Using this identity $Q = A(M)$, as well as the definitions of $\rho$ and $\psi$, we can write
\[
u\p_x M  = \rho \psi - \rho \Phi*\rho = \p_x Q - \p_x M\cdot (\phi*M)  = \p_x (A(M)) - \p_x M \cdot (\phi*M),
\]
which leads to the scalar balance law \eqref{e:balancelaw}.

\begin{REMARK}
\label{r:conv}
We include a shift of $-\tfrac12$ in the definition of $M$ in \eqref{eq:MQ} to make sense of the convolution $\phi*M$. We will always work with $\rho$ such that $\rho(t)$ is supported in a compact interval $[-R(T), R(T)]$ for any $t\in [0,T]$, so that $M(\pm x,t)=\pm\frac12$ for any $x\geq R(T)$. For $R\ge R(T)$, we define
\begin{equation}\label{eq:phiM}
\phi*M(x,t) = \int_{-2R}^{2R} \phi(z)M(x-z,t)\dz,\quad\text{for all}\,\,x\in [-R,R].
\end{equation}
Since $\phi$ is even, the choice of $R$ is inconsequential to the value of $\phi*M$; consequently, the definition \eqref{eq:phiM} defines $\phi*M(x,t)$ for all $x\in \R$, $t\in [0,T]$. Moreover,
\[
\phi*M(x,t) = \Big.\Phi(y)M(x-y)\Big]_{y=-2R}^{y=2R}+\int_{-2R}^{2R} \Phi(y)\rho(x-y,t)\dy = \Phi*\rho(x,t),
\]
for all $x\in [-R,R]$. Both boundary terms take the value $\tfrac12\Phi(2R)$ and hence cancel with each other.

Finally, it is easy to see that $\phi\ast M$ is bounded (by $\Phi(2R)$) and continuous (since $\Phi$ is).  
\end{REMARK}

\subsection{Entropy conditions to the scalar balance law}
The scalar balance law \eqref{e:balancelaw} admits a natural admissibility criterion via entropy inequalities. Let $\eta:[-\tfrac12,\tfrac12]\to \R$ be a convex and Lipschitz function and suppose $q:[-\tfrac12,\tfrac12]\to \R$ satisfies $q' = \eta' A'$. The pair $(\eta, q)$ is known as an entropy/entropy-flux pair, and the entropy inequality for \eqref{e:balancelaw} associated to each such pair reads
\begin{equation} 
\label{e:entropyineq}
\p_t\big(\eta(M)\big) + \p_x\big(q(M)\big) \le (\phi*M)\,\p_x\big(\eta(M)\big),
\end{equation}
in the sense of distributions. That is, for any nonnegative test function $\zeta\in C_c^\infty(\R\times (0,T))$, we require
\begin{equation}
\label{e:entropyineqdist}
\int_0^{T} \int_\R \Big[\eta(M) \p_t \zeta + q(M)\p_x \zeta + \zeta (\phi*M)\p_x (\eta(M)) \Big] \dx \dt
\ge 0.
\end{equation}

Let us comment on the last term in \eqref{e:entropyineqdist}. When $\phi$ is bounded and Lipschitz, it is easy to check that $\phi\ast M$ is Lipschitz. Therefore, the integral can be realized as
\[
\int_0^{T} \int_\R \zeta (\phi*M)\p_x (\eta(M)) \dx \dt = -
\int_0^{T} \int_\R \Big[\p_x\zeta (\phi*M) + \zeta \p_x(\phi*M)\Big] \eta(M) \dx \dt.
\]
For general locally integrable communication protocol $\phi$, note that $\eta(M)\in BV_{\loc}(\R)$ (e.g. by Lemma~\ref{l:BVchainrule}) and is constant outside of $[-R,R]$. Hence, $\p_x(\eta(M))$ is a Radon measure with support in $[-R,R]$.  Since $\phi\ast M\in C_b([-R,R])$, it follows that the integral is well-defined.

\begin{DEF}
\label{d:entropysoln}
We say $M : \R \times [0,T] \to [-\frac12,\frac12]$ is an \textit{entropy
  solution} to the scalar balance law \eqref{e:balancelaw} if
\begin{itemize}
  \item The entropy inequality \eqref{e:entropyineq} is satisfied for
    every entropy/entropy flux pair $(\eta, q)$.
  \item $M(\cdot,t)$ is nondecreasing, for any $t\in[0,T]$.
  \item There exists an $R(T)>0$, such that $M(\pm
    x,t)=\pm\frac12$, for any $x\ge R(T)$ and $t\in[0,T]$.
 \end{itemize}
We say $M : \R \times [0,+\infty) \to [-\frac12,\frac12]$ is an entropy solution if its restriction to any compact time interval $[0,T]$ is an entropy solution in the sense above.
\end{DEF}

By a standard approximation argument, one deduces that the collection of all entropy/entropy flux pairs in Definition \ref{d:entropysoln} may be replaced a smaller class. The entropy solution can be equivalently defined if the entropy inequality \eqref{e:entropyineq} is satisfied for every \emph{Kruzkov entropy pair} 
\begin{equation}
  \label{e:Kruzkovpair}
  \eta(m) = |m - \a|, \qquad q(m) = \sgn{m - \a}(A(m) - A(\a)), \qquad
  \a\in [-\tfrac12,\tfrac12].
\end{equation}
We will use the definition \eqref{e:Kruzkovpair} in the proofs of existence and uniqueness of entropy solutions to \eqref{e:balancelaw}.

Next, we state the Rankine--Hugoniot condition and the Oleinik entropy condition for \eqref{e:balancelaw}. Since the nonlocal term on the right-hand side of \eqref{e:entropyineq} plays a role in the conditions, we shall outline the derivation.

Suppose $M$ takes the values $M_\ell$ and $M_r$
on two sides of a shock along a curve $C = \{(x,t):x = s(t)\}$.  We denote by $\s(t) = \dot{s}(t)$ the speed of the shock. The entropy condition \eqref{e:entropyineqdist} becomes
\[
  \int_C \bigg([[\eta(M)]]\nu_t(s) + \Big([[q(M)]] - (\phi\ast M)[[\eta(M)]]\Big)\nu_x(s)\bigg)\zeta\ds 
  \ge 0,
\]
where $\nu = (\nu_x, \nu_t) = (1+\s^2)^{-\frac12}(1,-\s)$ is the unit normal vector along $C$, and we have used the notation $[[\eta(M)]]=\eta(M_\ell) - \eta(M_r)$ and $[[q(M)]]=q(M_\ell) - q(M_r)$. This implies
\[ 
(\s + (\phi*M))[[\eta(M)]] \le [[q(M)]],\qquad \text{along}\,\,C.
\]
Taking $\eta = \id$ and $q = A$, and noting that equality should hold
in \eqref{e:entropyineqdist}, we get the Rankine--Hugoniot condition
\begin{equation}
\label{e:RH}
\s +\phi*M = \frac{[[A(M)]]}{[[M]]}.
\end{equation}
For $M_\ell<M_r$, we take
$\eta(m) = (m-\theta) H(m-\theta)$ and $q(m)=(A(m)-A(\theta))H(m-\theta)$
for $\theta\in(M_\ell, M_r)$. Here, $H$ denotes the Heaviside function.
This leads to the \emph{Oleinik entropy condition}
\begin{equation}
\label{e:Oleinik}
\s + \phi*M\le \frac{A(\theta) - A(M_\ell)}{\theta - M_\ell},
\qquad \theta\in (M_\ell, M_r).
\end{equation}

\section{The sticky particle Cucker--Smale dynamics}
\label{s:Sticky}


\subsection{Definitions and notation}
\label{ss:DefNotSP}

Consider a system of particles of positive masses $(m_i)_{i=1}^N$ and
configurations $(x_i(t), v_i(t))_{i=1}^N$ following the Cucker--Smale
dynamics \eqref{e:CS} in one dimension
\begin{equation}
\label{eq:CS}
\frac{\dd x_i}{\dd t} = v_i,\qquad
\frac{\dd v_i}{\dd t} =  \sum_{\substack{j=1\\x_j\ne x_i}}^N m_j \phi(x_j - x_i) (v_j - v_i).
\end{equation}
We always assume the total mass is 1:
\begin{equation}
\label{e:mass1discrete}
\sum_{i=1}^N m_i = 1.
\end{equation}
The system \eqref{eq:CS}--\eqref{e:mass1discrete} allows particles to pass through each other.  We propose a modified system: the \textit{sticky particle Cucker--Smale dynamics}, which follows \eqref{eq:CS} except at times when two or more particles \emph{collide}, i.e., occupy the same position (for the first time). We insist that colliding particles remain stuck together for all future times.  We refer to the unmodified system \eqref{eq:CS}--\eqref{e:mass1discrete} as the \textit{Cucker--Smale dynamics without collisions}, in order to emphasize the distinction between the two sets of dynamics.

Let us now define precisely the collision rules for the sticky particle Cucker--Smale dynamics.  We fix the notation $J_i(t)$ to represent the set of indices $j$ such that particle $j$ is stuck to particle $i$ at time $t$:
\[ 
J_i(t):=\big\{j\in 1, \ldots, N:x_j(t) = x_i(t)\big\}. 
\]
\medskip
The collision rules have two ingredients:
\begin{itemize}
  \item Each collision is \emph{completely inelastic}, and particles stick to each other after collisions:
    \begin{equation}
      \label{e:stick}
      J_i(t) \supseteq  J_i(s), \quad \text{ whenever }\,\, t\ge s\ge 0;
    \end{equation}
  \item The collision conserves momentum:
    \begin{equation}
      \label{e:momentumconscollision}
      v_i(t+) = \frac{\sum_{j\in J_i(t)} m_j v_j(t-)}{\sum_{j\in J_i(t)} m_j}.
    \end{equation}
  \end{itemize}
We will always index the particles in increasing order from left to right:
\[
x_1(t)\le x_2(t)\le \cdots \le x_N(t).
\]
Since the collision rules do not allow particles to cross each other, the order stays unchanged for all time.  It will be convenient to set notation for lowest and highest indices in a given $J_i(t)$:
\[
i_*(t) = \min J_i(t), 
\qquad 
i^*(t) = \max J_i(t).
\] 

A time $t$ is called a \emph{collision time} if the cardinality of one or more of the $J_i$'s increases at time $t$.  At collision times, we make the convention that the $v_i$'s are right continuous, i.e., $v_i(t) = v_i(t+)$.  It is clear that at most $N-1$ collisions can occur.  Thus global-in-time existence and uniqueness of the sticky particle Cucker--Smale dynamics is a triviality.  

\subsection{Basic properties for the sticky particle Cucker--Smale dynamics}
We begin by stating the maximum principle on $(v_i)_{i=1}^N$, which is well-known for the Cucker--Smale dynamics without collisions.

\begin{PROP}
  Suppose $(x_i(t),v_i(t))_{i=1}^N$ follows the sticky particle Cucker--Smale dynamics associated to the data $(x_i^0,v_i^0, m_i)_{i=1}^N$. Then
  \begin{equation}
  \label{e:MPv}
    \min_{1\le j\le N} v_j(s) \leq  v_i(t)\leq \max_{1\le j\le N} v_j(s),\quad
    \text{for any}\,\, i=1,\ldots,N,\quad t\ge s\ge 0.
  \end{equation}
  Consequently, if $\{x_i^0\}_{i=1}^N\subset [-R^0,R^0]$, then
  \begin{equation}\label{e:xbound}
    \{x_i(t)\}_{i=1}^N\subset [-R(t),R(t)],\quad R(t):=R^0+
    t\max_{1\le j\le N}|v_j^0|,\quad \text{for all}\,\,t\geq0.
  \end{equation}
\end{PROP}
\begin{proof}
On collisionless time intervals, the bounds \eqref{e:MPv} follow immediately from the dynamics of $v_i$ in \eqref{eq:CS}. It is also clear that the collision rule \eqref{e:momentumconscollision} respects the maximum principle \eqref{e:MPv}.  The assertion \eqref{e:xbound} follows immediately from \eqref{e:MPv}.  
\end{proof}

\begin{REMARK}
A nondegenerate alignment forcing in the velocity equation will often decrease the best possible $R(t)$ in \eqref{e:xbound}.  In fact, under appropriate assumptions on the size and support of $\phi$, the radius $R(t)$ may be chosen independent of $t$, and each $v_i(t)$ will converge to the average $\bar{v}=\sum_i m_i v_i^0$ as $t\to\infty$.  These phenomena are known as \emph{flocking} and \emph{velocity alignment} and are well-known for the Cucker--Smale dynamics without collisions, c.f. \cite{HaLiu2008}.  We will discuss these properties in the context of the sticky particle Cucker--Smale dynamics and the Euler-alignment system in Section \ref{s:MoreProps}.
\end{REMARK}

We now consider some properties of the sticky particle Cucker--Smale dynamics that are unique to the 1D setting.  These properties are connected to the quantity
\begin{equation}\label{eq:psiit}
  \psi_i(t) = v_i(t) + \sum_{j=1}^N m_j \Phi(x_i(t) - x_j(t)),
\end{equation}
where $\Phi$ is the odd antiderivative of $\phi$, defined in \eqref{e:Phidef}. The quantity $\psi_i$ is a discrete analog of the macroscopic $\psi$ from  \eqref{e:psidef}. It plays the same role in many respects that the particle velocities $v_i$ do in the absence of alignment force. Indeed, $\psi_i = v_i$ in the degenerate case where $\phi \equiv 0$. The quantities $\psi_i$ have been used to study the 1D Cucker--Smale dynamics without collisions \cite{HaParkZhang2018,HaKimParkZhang2019}.  In the 1D sticky particle setup, they play a crucial role in the analysis of collisions.  We list two basic properties below.

\begin{PROP}
  The quantities $\psi_i$ have the following properties.
  \begin{itemize}
  \item [(a)] Each $\psi_i(t)$ is constant in time in the absence of collisions; for any $t$ that is not a collision time, 
    \begin{equation}
      \label{e:psiidot=0}
      \frac{\dd}{\dt}\psi_i(t) = 0.
    \end{equation}
  \item [(b)]  If $t$ is a collision time, we have
    \begin{equation}
      \label{e:psicollision}
      \psi_{i}(t+) = \frac{\sum_{j\in J_i(t)} m_{j}\psi_{j}(t-)}{\sum_{j\in J_i(t)} m_j}.
    \end{equation}
      \end{itemize}
\end{PROP}

Equation \eqref{e:psiidot=0} is verified by differentiating the formula for $\psi_i(t)$ and using \eqref{eq:CS}. It is a discrete analog of \eqref{eq:psi}. To prove \eqref{e:psicollision}, we apply the collision rule \eqref{e:momentumconscollision} and use the continuity of the trajectories $x_i(t)$:
\begin{align*}
\psi_i(t+)
& = v_i(t+) + \sum_{k=1}^N m_k \Phi(x_i(t) - x_k(t)) \\
& = \frac{\sum_{j\in J_i(t)} m_j v_j(t-)}{\sum_{j\in J_i(t)} m_j}
+ \frac{\sum_{j\in J_i(t)} m_j \left(  \sum_{k=1}^N m_k \Phi(x_j(t) - x_k(t)) \right)}{\sum_{j\in J_i(t)} m_j} 
= \frac{\sum_{j\in J_i(t)} m_j\psi_j(t-)}{\sum_{j\in J_i(t)} m_j}.
\end{align*}

\subsection{The Barycentric Lemma}

We are now in a position to state the following \emph{barycentric} lemma, which is the key to connecting the sticky particle dynamics with the entropy solutions of \eqref{e:balancelaw}. A similar argument has been implemented in \cite[Lemma 2.2]{brenier1998sticky} for the pressureless Euler equation, when $\phi \equiv 0$.  

\begin{LEMMA}
  \label{c:barycentric}
  Fix an $i\in \{1, \ldots, N\}$ and a time $t>0$. For any $k\in
  J_i(t)$, we have
  \begin{equation}
    \label{e:barycentric}
    \frac{\sum_{j=i_*(t)}^k m_j \psi_j(t-)}{\sum_{j=i_*(t)}^k m_j} \ge \frac{\sum_{j\in J_i(t)}  m_j \psi_j(t-)}{\sum_{j=J_i(t)} m_j} = \psi_i(t+) \ge \frac{\sum_{j=k}^{i^*(t)} m_j \psi_j(t-)}{\sum_{j=k}^{i^*(t)} m_j}.
  \end{equation}
\end{LEMMA}
\begin{proof}
  It suffices to establish the following monotonicity property:
  \begin{equation}
    \label{e:collisiondecri}
    \psi_{i_*(t)}(t-) \ge \psi_{i_*(t)+1}(t-) \ge \cdots \ge \psi_{i^*(t)}(t-).
  \end{equation} 
However, \eqref{e:collisiondecri} follows directly from the corresponding obvious monotonicity property for the velocities  
 \[
 v_{i_*(t)}(t-) \ge v_{i_*(t)+1}(t-) \ge \cdots \ge v_{i^*(t)}(t-),
 \]
 after taking into account that $\psi_j(t-) - v_j(t-)$ is independent of $j\in J_i(t)$.  
\end{proof}

\section{Entropy solutions to the discretized balance law}
\label{s:discbalancelaw}

In this section, we study the following discretized scalar balance law 
\begin{equation}
\label{e:MN}
\p_t M_N + \p_x(A_N(M_N)) = (\phi*M_N) \p_x M_N, \qquad
M_N(x,0)=M_N^0(x).
\end{equation}
as a first step toward understanding \eqref{e:balancelaw}. Equations \eqref{e:MN} and \eqref{e:balancelaw} are identical except for notation; we write \eqref{e:MN} separately to highlight the special discretized initial data $M_N^0$ and flux $A_N$ under consideration.

Let us describe our hypotheses. We assume $M_N^0$ is piecewise constant, of the form
\begin{equation}
\label{eq:MN0}
M_N^0(x) = -\frac12 + \sum_{j=1}^{N} m_j H(x - x_j^0),
\end{equation}
where the $m_j$'s are all strictly positive and sum to unity, as in \eqref{e:mass1discrete}.  We also assume $x_1^0\le x_2^0\le \cdots \le x_N^0$, and we use $H$ to denote the right-continuous Heaviside function, with $H(0) =1$.  
Note that the range of $M_N$ is discrete, consisting of the values $(\th_i)_{i=0}^N$ defined by
\begin{equation}\label{eq:theta}
\th_i := -\frac{1}{2}+\sum_{j=1}^i m_j.
\end{equation}
We define $A_N$ as a continuous and piecewise linear function, with breakpoints only at $(\th_i)_{i=1}^{N-1}$:
\begin{equation}\label{eq:AN}
A_N:[-\tfrac12,\tfrac12]\to~\R,\quad
  \text{$A_N$ is linear in each interval $[\th_{i-1}, \theta_i]$, for any
  $i=1,\ldots,N$.}
\end{equation}

Our main purpose in this section is to demonstrate that one can generate an entropy solution to \eqref{e:MN} using the sticky particle Cucker--Smale dynamics. This builds a connection between the collision rules \eqref{e:stick}--\eqref{e:momentumconscollision} and the entropy conditions for the scalar balance law \eqref{e:balancelaw}. This connection will be further developed later into a selection principle for a unique weak solution of the 1D Euler-alignment system.

\begin{THEOREM}
\label{t:discr}
Consider the scalar balance law \eqref{e:MN} with discrete initial data $M_N^0$ and flux $A_N$, satisfying the hypotheses \eqref{eq:MN0} (for some $m_i$'s and $x_i^0$'s as described above) and \eqref{eq:AN}, respectively. For each $i=1, \ldots, N$, define $\psi_i^0$ to be the slope of the $i$-th piece of $A_N$, 
\begin{equation}
\label{eq:psiAn}
\psi_i^0 =  A_N'(m), \qquad \text{for}\,\,\,\th_{i-1}<m<\th_i,
\end{equation}
and put 
\begin{equation}
\label{e:SPCSapprox0}
v_i^0 = \psi_i^0 - \sum_{j=1}^{N} m_j \Phi(x_i^0 - x_j^0).
\end{equation}
Let $(x_i(t),v_i(t))_{i=1}^{N}$ follow the sticky particle Cucker--Smale dynamics associated to the masses $(m_i)_{i=1}^{N}$ and the initial conditions $(x_i^0,v_i^0)_{i=1}^{N}$.  Then 
\begin{equation}
\label{e:MNsoln}
M_N(x,t) = -\frac12+\sum_{i=1}^{N} m_i H(x-x_i(t))
\end{equation}
is an entropy solution of the discretized balance law \eqref{e:MN}. Moreover, we have
\begin{equation}\label{e:QN}
A_N\circ M_N (x,t) = A_N(-\tfrac12)+\sum_{i=1}^{N} m_i \psi_i(t)H(x-x_i(t)).
\end{equation}
\end{THEOREM}

\begin{proof}
Since $M_N(\cdot, t)$ in \eqref{e:MNsoln} is piecewise constant, it suffices to check that the shock discontinuities along the curves $C_i = \{(x_i(t), t):t\ge 0\}$ satisfy the Rankine--Hugoniot condition \eqref{e:RH} and the Oleinik entropy condition \eqref{e:Oleinik}, with the shock speed $\s_i(t)=v_i(t)$.

Fix a point $(x_i(t), t)$ on $C_i$. By definition \eqref{e:MNsoln}, we get 
\[
M_N(x_i(t)-,t) = \th_{i_*(t)-1},\quad M_N(x_i(t)+,t)=M_N(x_i(t),t)=\th_{i^*(t)}.
\]
We denote the jump of a function $f$ across $C_i$ by $[[f]] = f(x_i(t)+) - f(x_i(t)-)$. Thus
\[
 [[M_N(\cdot, t)]] = \th_{i^*(t)}-\th_{i_*(t)-1}=\sum_{j\in J_i(t)} m_j,
\]
and from \eqref{eq:psiAn}
\[
[[A_N\circ M_N(\cdot, t)]]  = \int_{\th_{i_*(t)-1}}^{\th_{i^*(t)}} A_N'(m)\dd m = \sum_{j\in J_i(t)} \int_{\th_{j-1}}^{\th_j} A_N'(m)\dd m = \sum_{j\in J_i(t)} m_j \psi_j^0.
\]
Applying \eqref{e:psicollision}, we verify the Rankine--Hugoniot condition \eqref{e:RH}
\[
\frac{[[A_N\circ M_N(\cdot, t)]]}{[[M_N(\cdot, t)]]} = \frac{\sum_{j\in J_i(t)} m_j \psi_j^0}{\sum_{j\in J_i(t)} m_j} = \psi_i(t) = v_i(t) + \phi*M_N(t), \qquad \text{ along }C_i.
\]

Next, we check the Oleinik entropy condition \eqref{e:Oleinik}, that is,
\[
v_i(t) + \phi*M_N(x_i(t),t) \le \frac{A_N(\theta) - A_N(\th_{i_*(t)-1})}{\theta - \th_{i_*(t)-1}},
\qquad \theta\in (\th_{i_*(t)-1}, \th_{i^*(t)}).
\]
Since $A_N$ is piecewise linear, it suffices to check the inequality for $\th=\th_k$, $k\in \{i_*(t)-1,\ldots, i^*(t)\}$.  Applying Lemma \ref{c:barycentric}, we obtain 
\begin{align*}
\frac{A_N(\th_k) - A_N(\th_{i_*(t)-1})}{\th_k - \th_{i_*(t)-1}}=\frac{\sum_{j=i_*(t)}^k m_j \psi_j^0}{\sum_{j=i_*(t)}^k m_j}
\ge \frac{\sum_{j\in J_i(t)} m_j \psi_j^0}{\sum_{j\in J_i(t)} m_j} = \psi_i(t) = v_i(t) + \phi*M_N(x_i(t),t).
\end{align*}

Finally, we check \eqref{e:QN}. The equality is trivial when $x<x_1(t)$. For $x\ge x_1(t)$, let $i$ be the smallest index such that $x>x_i(t)$.  Then we have $M_N(x,t)=\th_i$, so that, recalling \eqref{eq:psiAn}, we get
\[
A_N(M_N(x,t))=A_N(\th_0)+\sum_{j=1}^i\int_{\th_{j-1}}^{\th_j}A_N'(m)\dd m
=A_N(-\tfrac12)+\sum_{j=1}^im_j\psi_j^0
=A_N(-\tfrac12)+\sum_{j=1}^im_j\psi_j^0.
\]
The conservation of momentum \eqref{e:psicollision} implies
$\sum_{j=1}^im_j\psi_j(t)=\sum_{j=1}^im_j\psi_j^0$, which ends the proof.
\end{proof}

\section{The scalar balance law}
\label{s:M}

In this section, we focus on developing global wellposedness theory
for the scalar balance law \eqref{e:balancelaw}, which we recall for the reader's convenience: 
\begin{equation}\label{eq:M}
\p_t M + \p_x(A(M)) = (\phi*M)\p_x M,\qquad M(x,0)=M^0(x).
\end{equation}

The existence and uniqueness theory for entropy solutions of scalar
conservation laws has been well-established. The additional feature of
\eqref{eq:M} is the right-hand side of the equation, which is both nonlinear and
nonlocal, requires extra treatment.
We show that the entropy solution of \eqref{eq:M}, in the sense of Definition
\ref{d:entropysoln}, exists and is unique. Furthermore, it can be
approximated by the sticky particle Cucker--Smale dynamics.
We also obtain stability bounds with respect to the initial condition
$M^0$, as well as the flux $A$.
Our main theorem is stated as follows.

\begin{THEOREM}
\label{t:M}
Consider the scalar balance law \eqref{eq:M}.   Assume the initial condition $M^0$ is a nondecreasing function and that there exists an $R^0>0$ such that $M^0(\pm x)=\pm\frac12$ for any $x\ge R^0$. Let the flux $A:[-\frac12,\frac12]\to \R$ be a Lipschitz function.
\begin{itemize}
\item [(a)] \emph{(Existence and Uniqueness)} Given any $T>0$, the Cauchy problem \eqref{eq:M} has a unique entropy solution
\[
M\in BV(\R\times[0,T]).
\]
\item [(b)] \emph{(Approximability)} For any $T>0$, the entropy solution $M$ of \eqref{eq:M} on $[0,T]$ can be approximated by the discretized balance law \eqref{e:MN}, and hence by the sticky particle Cucker--Smale dynamics, in the following sense. There exists a sequence of (explicit) discrete initial data $M_N^0$ and fluxes $A_N$, satisfying the hypotheses \eqref{eq:MN0} and \eqref{eq:AN}, respectively, such that the associated entropy solutions $M_N$ of \eqref{eq:M} satisfy
\begin{equation}\label{eq:Mlimit}
  M_N-M\to 0 \quad \text{in}\,\, C([0,T];L^1(\R)),
\end{equation}
and 
\begin{equation}\label{eq:Mtlimit}
\p_tM_N(\cdot,t)\stackrel{*}{\rightharpoonup}  \p_tM(\cdot,t)\quad \text{in}\,\, \cM(\R).
\end{equation}
for any $t\in[0,T]$. Here, $\cM$ is the space of measures.
\item [(c)] \emph{(Stability)} Let $\tM$ be the entropy solution of the scalar balance law
\[
\p_t \tM + \p_x(\wtA(\tM)) = (\phi*\tM)\p_x \tM,\quad \tM(x,0)=\tM^0(x)
\]
with initial data $\tM^0$ and flux $\wtA$ that satisfy the same assumptions as $M^0$ and $A$ respectively. Then for any $t\geq0$, we have the following stability bound:
\begin{equation}
\label{e:stability}
\|M(\cdot,t)-\tM(\cdot,t)\|_{L^1(\R)}
\le \|M^0-\tM^0\|_{L^1(\R)} + t |A - \wtA|_{\Lip} .
\end{equation}
	\end{itemize}	
\end{THEOREM}

The rest of the section is devoted to the proof of this theorem. Before beginning in earnest, however, we note the following.  In our argument, we will need to differentiate the
composition of a Lipschitz function and a $BV$ function.  To make
sense of such an operation, one can use Vol'pert's theory of the $BV$
calculus \cite{Volpert1967}. The precise version of the $BV$ chain
rule that we need is stated in \cite[Lemma A2.1]{BouchutPerthame1998}.
\begin{LEMMA}
	\label{l:BVchainrule}
	Suppose $W\in BV_{\loc}(\R)$ and $f$ is Lipschitz. Then
	$f\circ W$ belongs to $\in BV_{\loc}(\R)$, and in the sense of measures,
	\begin{equation}\label{eq:BVchain}
	\left| \frac{\dd}{\dx} (f\circ W) \right| \le |f|_{\Lip} \left|\frac{\dd}{\dx} W\right|. 
	\end{equation}
\end{LEMMA}

\subsection{Existence and approximability}

\label{ss:existence}

We start with the existence part of Theorem \ref{t:M}(a). The plan is to construct an entropy solution of \eqref{eq:M} using the front-tracking scheme, c.f. \cite[Chapter 14]{Dafermosbook}. A front-tracking approximation of \eqref{eq:M} follows precisely the dynamics of the discretized balance law \eqref{e:MN}; therefore, we will construct a sequence of approximated solutions $M_N$, extract a limit $M$, and show that $M$ is an entropy solution of \eqref{eq:M}.

\subsubsection*{Step 1: Constructing an approximation sequence}
For a given $N$, we construct an initial condition $M_N^0$ and flux $A_N$ for the discretized balance law \eqref{e:MN}. As long as the hypotheses \eqref{eq:MN0} and \eqref{eq:AN} are satisfied, we can apply Theorem \ref{t:discr} to get a solution $M_N$ of the form \eqref{e:MN}.  We give slightly more detail in this step than what is strictly necessary for the proof of Theorem \ref{t:M}(a); we do this to allow for the reader to easily compare the approximation scheme we use here with the one we use later in Theorem \ref{t:approxrate}.

We begin with an $N$-tuple of positive masses $(m_{i,N})_{i=1}^N$ that sum to unity \eqref{e:mass1discrete}. We also assume
\begin{equation}\label{eq:miN}
  \lim_{N\to\infty}\max_{1\le i\le N}m_{i,N}=0.
\end{equation}
A typical choice is $m_{i,N} = \tfrac1N$, so all particles have the same mass. Next, we define $x_{i,N}^0$ by
\begin{equation}\label{eq:xiN0}
  x_{i,N}^0 = \inf\big\{x:M^0(x)\ge \th_{i,N}\big\},\quad
  i=1,\ldots,N.
\end{equation}
where $\th_{i,N}$ is defined in \eqref{eq:theta}.
It is easy to check that $\{x_{i,N}^0\}_{i=1}^N\subset [-R^0,R^0]$.
Then, $M_N^0$ can be constructed from \eqref{eq:MN0}. 
Finally, we define $A_N$ as the piecewise linear approximation of $A$
such that
\begin{equation}\label{eq:ANdef}
  A_N(\th_{i,N})=A(\th_{i,N}),\quad i=0,\ldots,N.
\end{equation}
 
The $M_N^0$ and $A_N$ constructed through the procedure above clearly
satisfy the hypotheses \eqref{eq:MN0} and \eqref{eq:AN}. Moreover,
they approximate $M^0$ and $A$ in the following
sense.

\begin{LEMMA}\label{lem:approx0}
The following inequalities hold:
\begin{equation}\label{eq:MAroughlimit}
\|M_N^0 - M^0\|_{L^1(\R)} \le 2R^0 \max_{1\le i\le N} m_{i,N},
\qquad 
\sup_{m\in [-\frac12, \frac12]} |A_N(m) - A(m)| \le |A|_{\Lip} \max_{1\le i\le N} m_{i,N}.
\end{equation}
In particular, $M_N^0- M^0\to 0$ in $L^1(\R)$ and $A_N\to A$ uniformly, as $N\to \infty$.
\end{LEMMA}
\begin{proof}
  Denote $x_{0,N}^0:=-R^0$.
  We have
  \[\|M_N^0-M^0\|_{L^1(\R)}=\sum_{j=1}^N\int_{x_{j-1,N}^0}^{x_{j,N}^0}\big(M^0(x)-M_N^0(x)\big)\dx\leq
    \sum_{j=1}^N m_{j,N}(x_{j,N}^0-x_{j-1,N}^0)\leq 2R^0\max_{1\le j\le N} m_{j,N},\]
  which proves the first inequality in \eqref{eq:MAroughlimit}.  Note
  that we may allow $x_{j-1,N}^0=x_{j,N}^0$ for some $j$'s, and the
  estimate above still holds.  As for the second inequality in \eqref{eq:MAroughlimit}, fix $m\in [-\frac12, \frac12]$ and choose $i$ such that  $m\in[\th_{i-1,N},\th_{i,N}]$.  Then \eqref{eq:AN} and \eqref{eq:ANdef} imply 
  \[
  A_N(m) - A(m) = \frac{m-\th_{i-1,N}}{m_i}(A(\th_{i-1,N}) - A(m)) + \frac{\th_{i,N}-m}{m_i}(A(\th_{i,N}) - A(m)),
  \]
  which easily implies the second inequality.
\end{proof}
	
\subsubsection*{Step 2: Extracting a limit $M$}
Fix a time $T>0$. For any $t\in[0,T]$, since $M_N(t)$ is uniformly bounded and nondecreasing, we  may apply Helly's selection theorem and find a convergent subsequence $M_{N_k}(t)$ in $L^1_{\loc}(\R)$. Using a diagonal argument, we can get a further subsequence, still denoted by $M_{N_k}$, that is convergent for all rational $t\in[0,T]$ in $L^1_{\loc}(\R)$.  We provisionally denote the limit by $M(t)$.  We want to upgrade the convergence $M_{N_k}(t)-M(t)\to 0$ from $L^1_{\loc}(\R)$ to $L^1(\R)$ and also extend our conclusion to irrational times.  The following observation will help us achieve this.  

Note that by 
\eqref{eq:psiAn}, \eqref{e:SPCSapprox0}, and the monotonicity of $\Phi$, we have the following $N$-independent bound on the initial velocities $v_{i,N}^0$:
\begin{equation}
\label{e:viN0bd}
|v_{i,N}^0| = \left|\psi_{i,N}^0 - \sum_{j=1}^{N} m_{j,N}
\Phi(x_{i,N}^0 - x_{j,N}^0)\right|\le |A_N|_{\Lip}  + \Phi(2R^0) \le |A|_{\Lip} +  \Phi(2R^0).
\end{equation}
Then, \eqref{e:xbound} implies that for $t\in [0,T]$, we have $\{x_{i,N}(t)\}_{i=1}^N\subset [-R(T),R(T)]$, where 
\begin{equation}\label{eq:RT}
  R(T)=R^0+T(|A|_{\Lip} +  \Phi(2R^0)).
\end{equation}
It follows that $M_N(\pm x,t)=\pm\tfrac12$ for all $x>R(T)$ and $t\in[0,T]$, and thus we have $M_{N_k}(t)-M(t)\to 0$ in $L^1(\R)$ for all rational times $t\in \Q_+$.  The extension of this  convergence to irrational times is an easy consequence of the time regularity estimate
\begin{align}\label{eq:timereg}
  \int_\R |M_N(x,t) - M_N(x,s)|\dx
\le &\,\sum_{i=1}^Nm_i|x_{i,N}(t)-x_{i,N}(s)|
\le \max_{i} |v_{i,N}^0|\cdot (t-s)\\  \le&\, \big(|A|_{\Lip} +  \Phi(2R^0)\big) (t-s).\nonumber
\end{align}
We used \eqref{e:MNsoln} to get the first inequality, then the maximum principle \eqref{e:MPv} to get the second, and finally the bound \eqref{e:viN0bd} to finish.  

Combining \eqref{eq:timereg} with the established convergence $M_{N_k}(t)-M(t)\to 0$ at rational times, we conclude
\[
M_{N_k}-M\to 0\quad \text{in}\,\, C([0,T];L^1(\R)).
\]
The limit $M$ has the desired properties: For each $t\in [0,T]$, the function $M(\cdot, t)$ is nondecreasing, with $M(\pm x,t)=\pm\tfrac12$ for all $x\ge R(T)$. Moreover, the time regularity estimate
\eqref{eq:timereg} implies a uniform bound
\[\left\|\p_tM_{N_k}(x,t)\right\|_{\cM}\leq |A|_{\Lip} +  \Phi(2R^0),\quad \text{for all}~~t\in[0,T].\]
Then, extracting a further subsequence, still denoted by $M_{N_k}$, we
obtain the weak-$*$ convergence
\[\p_tM_{N_k}(\cdot,t)\stackrel{*}{\rightharpoonup}  \p_tM(\cdot,t)\quad \text{in}\,\, \cM(\R).\] 
This also allows us to conclude that $M\in BV(\R\times[0,T])$.

\begin{REMARK}
Once we show $M$ is the \emph{unique} entropy solution of \eqref{eq:M} (through an argument independent of the existence proof), we can conclude that the whole sequence $M_N$ converges to $M$, finishing the proof of Theorem \ref{t:M}(b). 
\end{REMARK}
	
\subsubsection*{Step 3: Verifying the entropy conditions} Finally, we show that the function $M$ we have constructed above is indeed an entropy solution of \eqref{eq:M}.  We do this by verifying the entropy inequality \eqref{e:entropyineqdist} for all Kruzkov entropy pairs $(\eta, q)$ in \eqref{e:Kruzkovpair}.

We know from Theorem \ref{t:discr} that $M_N$ is an entropy solution of \eqref{e:MN}. Thus the entropy inequality \eqref{e:entropyineqdist} is satisfied for $(\eta, q_N)$, where $\eta(m)=|m-\a|$ and $q_N(m)=\sgn{m-\a}(A_N(m)-A_N(\a))$. It reads
\[
\int_0^{T} \int_\R \Big[\eta(M_N) \p_t \zeta + q_N(M_N)\p_x \zeta +
(\phi*M_N)\zeta \p_x(\eta(M_N)) \Big] \dx \dt
\ge 0.
\]
Now we pass to the limit. To simplify the notation, we write $M_N$
instead of $M_{N_k}$ in what follows. Define $q(m) = \sgn{m-\a}(A(m) -
A(\a))$, as in \eqref{e:Kruzkovpair}.  For a fixed $t\in[0,T]$ (for
which we suppress the notation), we use Lemma \ref{lem:approx0} and get
\begin{align*}
  \|\eta(M_N)-\eta(M)\|_{L^1}\leq
  &\,|\eta|_{\Lip}\|M_N-M\|_{L^1}=\|M_N-M\|_{L^1}\to0,\\
  \|q_N(M_N)-q(M)\|_{L^1}\leq &\,\|q_N(M_N)-q(M_N)\|_{L^1}+\|q(M_N)-q(M)\|_{L^1}\\
  \leq &\, 2R(T)\|A_N-A\|_{L^\infty([-\frac12,\frac12])}+|A|_{\Lip}\|M_N-M\|_{L^1}\to0.
\end{align*}
This establishes convergence for the first two terms.  The last term is more subtle; we argue as follows.

First, we claim that $\phi*M_N$ converges uniformly to $\phi*M$ with respect to $x$, and that this convergence is furthermore uniform with respect to $t$ on $[0,T]$.  This is immediate if $\phi$ is bounded; otherwise we can consider a mollification $\phi_\d$ of $\phi$ and estimate as follows:
\begin{align*}
\|\phi*M_N - \phi*M\|_{L^\infty} 
& \le \|\phi - \phi_\d\|_{L^1} \|M_N - M\|_{L^\infty} + \|\phi_\d\|_{L^\infty} \|M_N - M\|_{L^1}.
\end{align*}
We can first choose $\d$ so that the first term on the right is small, then choose $N$ large enough to finish.

Next, we note that (by  Lemma \ref{l:BVchainrule}) $\p_x(\eta(M_N))$ is a bounded sequence in $\cM(\R)$:  
\[
\|\p_x(\eta(M_N))\|_{\cM} \le |\eta|_{\Lip}.
\]
The same bound holds for $\p_x(\eta(M))$.  Now we can write 
\begin{align*}
& \left| \int_0^{T} \int_\R (\phi*M_N)\zeta \p_x (\eta(M_N)) \dx \dt - \int_0^{T} \int_\R (\phi*M_N)\zeta \p_x (\eta(M_N)) \dx \dt \right| \\
& \le \int_0^{T}  \|\zeta(t)\|_{L^\infty}\|\phi*M_N - \phi*M\|_{L^\infty} \|\p_x(\eta(M_N))\|_{\cM} \dt \\
& \qquad + 
\left| \int_0^{T} \int_\R (\phi*M)\zeta \big[ \p_x(\eta(M_N)) - \p_x(\eta(M))\big] \dx\dt \right|.
\end{align*}
The first term on the right side of the above inequality goes to zero in light of the above arguments. We can establish the vanishing of the second term similarly:  Mollifying $\phi*M$ if necessary, we write 
\begin{align*}
& \left| \int_0^{T} \int_\R (\phi*M)\zeta \big[ \p_x(\eta(M_N)) - \p_x(\eta(M))\big] \dx\dt \right| \\
& \le \int_0^{T} \| \zeta (\phi*M - (\phi*M)_\d)\|_{L^\infty} \| \p_x(\eta(M_N)) - \p_x(\eta(M)) \|_\cM \dt \\
& \quad + \int_0^{T} \|\p_x\big( \zeta(\phi*M)_\d \big) \|_{L^\infty} \|\eta(M_N) - \eta(M)\|_{L^1} \dt.
\end{align*}
Note that mollification is unnecessary if $\phi$ is bounded.  In any case, the continuity of $\phi*M$ and the compact support of $\zeta$ guarantee that $\|\zeta\big((\phi*M) - (\phi*M)_\d \big) \|_{L^\infty}$ can be made as small as desired by choosing $\d$ appropriately, after which we can choose $N$ large enough to make $\|\eta(M_N) - \eta(M)\|_{L^1}$ small.  We thus obtain the entropy inequality \eqref{e:entropyineqdist} and conclude that $M$ is an entropy solution of \eqref{eq:M}.

\subsection{Uniqueness and $L^1$ stability}
\label{ss:uniqueness}

We now prove the stability estimate \eqref{e:stability}. Note that uniqueness is a direct consequence if we set $\tM^0=M^0$ and $\wtA=A$. We use Kruzkov's \emph{doubling of the variables} strategy, with additional treatment of the nonlocal term on the right-hand side of \eqref{eq:M}.

For fixed $(y,s)$, consider the Kruzkov entropy pair \eqref{e:Kruzkovpair} with $\a = \tM(y,s)$, and a test function  $\zeta(x,t) = w(x,t,y,s)$ to be specified later. Entropy inequality~\eqref{e:entropyineqdist}~reads
\begin{align*}
0 & \le  \iint |M(x,t) - \tM(y,s)|\p_t w(x,t,y,s)\dx\dt \\
& \qquad + \iint \sgn{M(x,t) - \tM(y,s)}(A(M(x,t))- A(\tM(y,s)))\p_x w(x,t,y,s) \dx \dt \\
& \qquad + \iint (\phi*M)(x,t) (\p_x |M(x,t) - \tM(y,s)|) w(x,t,y,s) \dx \dt.
\end{align*}
We omit the bounds of integration in most of the computation below. Unless otherwise specified, the spatial variables $x$ and $y$ are integrated over $\R$, while the time variables $t,s$ are integrated over $[0,T]$.

We perform the analogous manipulations, with $\widetilde{A}$ replacing
$A$ and the roles of  $M(x,t)$ and $\widetilde{M}(y,s)$ interchanged.
Integrating over the remaining free variables in both cases and adding
the results yields
\begin{align}
  \label{e:Mdoubled}
  0 & \le \iiiint |M(x,t) - \tM(y,s)|(\p_t w+\p_s w)(x,t,y,s)\dx\dt\dy\ds \\
  & \quad + \iiiint \sgn{M(x,t) - \tM(y,s)}\bigg[ (A(M(x,t))- A(\tM(y,s)))\p_x w(x,t,y,s) \nonumber\\
  & \hspace{62 mm} + (\wtA(M(x,t))- \wtA(\tM(y,s)))\p_y w(x,t,y,s)\bigg] \dx \dt \dy \ds \nonumber\\
  & \quad + \iiiint w(x,t,y,s) \bigg[ (\phi*M)(x,t)\p_x |M(x,t) - \tM(y,s)| \nonumber \\
  & \hspace{40mm} + (\phi*\tM)(y,s)\p_y |M(x,t) - \tM(y,s)|\bigg] \dx\dt\dy\ds. \nonumber
\end{align}

We introduce the auxiliary variables
\[
\overline{x} = \tfrac{x+y}{2}, 
\quad 
\overline{y} = \tfrac{x-y}{2},
\quad 
\overline{t} = \tfrac{t+s}{2},
\quad \text{and}\quad
\overline{s} = \tfrac{t-s}{2},
\]
and we take a test function of the form
\[
w(x,t,y,s) 
= b_\e\left( \tfrac{x-y}{2} \right) b_\e\left( \tfrac{t-s}{2} \right)g \left( \tfrac{x+y}{2} \right) h_\d \left( \tfrac{t+s}{2} \right) 
= b_\e(\bary)b_\e(\bars) g(\barx) h_\d(\bart),
\]
where $b_\e, g, h_\d$ are smooth, nonnegative functions satisfying the following properties.
\begin{itemize}
	\item The functions $(b_\e)_{\e>0}$ approximate the Dirac delta distribution as $\e\to 0+$. We take $b_\e$ to be a standard mollifier, supported in $(-\e,\e)$ and having integral $1$.  
	\item The function $g$ is identically $1$ on $[-R(T),R(T)]$ and is compactly supported.
	\item The functions $(h_\d)_{\d>0}$ approximate the indicator function of $[s,t]$ as $\d\to 0+$. We take $h_\d$ to be identically $1$ on $[s,t]$,          identically zero outside $[s-\d, t+\d]$, and linear on $[s-\d, s]$ and $[t,t+\d]$.   
\end{itemize}

To proceed, we shall substitute our test function into \eqref{e:Mdoubled}, using the auxiliary variables. Observe that
\[
\p_t  + \p_s  = \p_{\bart},\quad
\p_x  + \p_y  = \p_{\barx},\quad
\p_x  - \p_y  = \p_{\bary}.
\]
We now rewrite our inequality in terms of the new variables.  In particular, the bracketed part of the second term on the right side of \eqref{e:Mdoubled} can be rewritten as
\[ 
\Big(A_+(M)-A_+(\tM)\Big)\,\p_{\barx} w+\Big(A_-(M)-A_-(\tM)\Big)\,
  \p_{\bary} w,
\]
where we have used the shorthand notation
\[
  A_\pm(m)  := \frac{A(m)\pm \wtA(m)}{2}
\]  
and suppressed the arguments of $M$ and $\tM$.  The latter will be equal to $(x,t)=(\barx+\bary,\bart+\bars)$ and $(y,s)=(\barx-\bary,\bart-\bars)$, respectively, for the rest of the computation below.  

Note that if the fluxes $A$ and $\wtA$ are identical, then $A_+$ reduces to their common value, while $A_-$ vanishes.  Hence, $A_-$ encodes information about stability with respect to the flux.  

Next, we rewrite the bracketed part in the last term of \eqref{e:Mdoubled} (again suppressing arguments) as 
\[
  \frac{\phi*M + \phi*\tM}{2}\cdot \p_{\barx} |M - \tM|
+  \frac{\phi*M - \phi*\tM}{2} \cdot\p_{\bary} |M-\tM|.
\]

Substituting the above into \eqref{e:Mdoubled} then yields
 \begin{align}
   \label{e:Mdoubled2}
   0 \le & \iiiint \bigg[ |M - \tM|\, \p_{\bart}w + \sgn{M - \tM}(A_+(M)- A_+(\tM))\p_{\barx} w\bigg] \dbar\\
   & + \iiiint \sgn{M - \tM}(A_-(M)- A_-(\tM)) \p_{\bary}w \dbar \nonumber  \\
   & + \frac12 \iiiint (\phi*M + \phi*\tM)w\cdot \p_{\barx} |M - \tM| \dbar \nonumber \\
   & + \frac12 \iiiint (\phi*M - \phi*\tM)w\cdot \p_{\bary} |M - \tM| \dbar \nonumber
\end{align}

We want to take $\e\to 0$, which will effectively set $\bary$ and $\bars$ equal to zero.  Before we can do this, however, we need to deal with the $\bary$ derivatives.  We treat the second integrable above first, making use of the following lemma to justify the necessary integration by parts.  

\begin{LEMMA}\label{lem:gamma}
The function 
\[
\g(M,\tM):=\sgn{M - \tM}(A_-(M)-A_-(\tM))
\]
is Lipschitz in both variables $M$ and $\tM$, with
  \[|\g(\cdot,\tM)|_{\Lip}\leq |A_-|_{\Lip},\qquad
  |\g(M,\cdot)|_{\Lip}\leq |A_-|_{\Lip}.\]
\end{LEMMA}
\begin{proof}
Fix an $\tM$ and pick $M_1<M_2$.  We consider two cases. First, if $M_1$ and $M_2$ are both greater or both less than $\tM$, then
\[
|\g(M_1,\tM)-\g(M_2,\tM)|=|A_-(M_1)-A_-(M_2)|
\leq |A_-|_{\Lip}|M_1-M_2|.
\]
If on the other hand we have $M_1\le \tM \le M_2$, then
\[
|\g(M_1,\tM)-\g(M_2,\tM)| 
\le  |A_-|_{\Lip}(\tM-M_1)+ |A_-|_{\Lip}(M_2-\tM)
= |A_-|_{\Lip}|M_1-M_2|.
\]
The the estimate $|\g(\cdot,\tM)|_{\Lip}\leq |A_-|_{\Lip}$ follows. The other bound can be obtained in the same way.
\end{proof}
Now, we apply \eqref{eq:BVchain} with $f(z)=\gamma(z,\tM), W=M$ and $f(z)=\gamma(M,z)$, $W=\tM$.  Lemma \ref{lem:gamma} yields
\[
\left|\frac{\dd}{\dd\bary}\g(M,\tM)\right|
\leq |A_-|_{\Lip}\,\left|\p_1 M(\barx+\bary,\bart+\bars)\right|
+ |A_-|_{\Lip} \,\left|\p_1\tM(\barx-\bary,\bart-\bars)\right|. 
\]
Here $\p_1$ denotes differentiation with respect to the first (spatial) argument.

We also use the following estimate in the fourth integral of \eqref{e:Mdoubled2}:
\[
\left|\frac{\dd}{\dd\bary}|M  (\barx+\bary,\bart+\bars)-\tM
(\barx-\bary,\bart-\bars)|\right|\leq \left|\p_1M
(\barx+\bary,\bart+\bars) + \p_1\tM (\barx-\bary,\bart-\bars)\right|.
\]
We now collect all the estimates above and take $\e\to 0$ in \eqref{e:Mdoubled2}.  For simplicity, we revert to the notation $(x,t)$ in rather than $(\barx, \bart)$ in this inequality and the following ones.  We obtain
\begin{align*}
0 \le & \iint \bigg[ |M(x,t) - \tM(x,t)|\,g(x)h'_\d(t) + \sgn{M - \tM}(A_+(M)- A_+(\tM)) g'(x)h_\d(t) \bigg] \dx\dt \\
   & + \iint |A_-|_{\Lip} (|\p_x M|+|\p_x \tM|) g(x)h_\d(t) \dx \dt \\
   & + \frac12 \iint \bigg[ \phi*(M + \tM) \cdot \p_x |M - \tM| + (\phi*|M-\tM|)\cdot |\p_x M+\p_x \tM| \bigg] g(x)h_\d(t) \dx \dt. 
\end{align*}

Next, we recall our choices of $g$ and $h_\d$. We can drop the second term in the first integral above, since $g'\equiv 0$ in $[-R(T),R(T)]$; we also replace $g$ by $1$ for the rest of the terms. Taking $\d\to 0$, we get 
\begin{align}\label{e:Mdoubled3}
   & \int |M(x,t) - \tM(x,t)|\dx \\
   & \leq \int |M(x,s) - \tM(x,s)|\dx +
     \frac12 |A-\wtA|_{\Lip}
     \int_s^t \int |\p_x M| + |\p_x \tM| \dx\,\dd\tau \nonumber\\
   & \quad + \frac12 \int_s^t \int \bigg[ \Phi*(\p_x M + \p_x \tM) \cdot \p_x|M-\tM| + |\p_x M + \p_x\tM| (\Phi*\p_x|M-\tM|)\bigg] \dx\,\dd\tau.\nonumber
\end{align}
Up to this point, we have not used the fact that $M$ and $\widetilde{M}$ are nondecreasing.  We take advantage of it in this final step by replacing $|\p_x M|$ with $\p_x M$, etc.  Under these replacements, the second term in \eqref{e:Mdoubled3} becomes $|A-\wtA|_{\Lip}(t-s)$, while the last term in \eqref{e:Mdoubled3} vanishes identically due to the oddness of $\Phi$. The stability bound \eqref{e:stability} follows immediately, upon taking $s=0$.

\section{The entropic selection principle for the Euler-alignment System}

\label{s:EArecovery}

In this section, we come back to our main 1D Euler-alignment system \eqref{eq:1DEA}. Recall
\begin{equation}
\label{eq:EA}
\left\{
\begin{array}{rcl}
\p_t \rho + \p_x (\rho u) & = & 0, \\
\p_t (\rho u) + \p_x (\rho u^2) & = & \rho (\phi*(\rho u)) - \rho u (\phi*\rho),
\end{array}
\right.\qquad
\left\{
\begin{array}{l}
  \rho(x,0)=\rho^0(x),\\
  u(x,0)=u^0(x).
\end{array}
\right.
\end{equation}

We construct a uniquely determined weak solution of \eqref{eq:EA}, using the entropy conditions \eqref{e:entropyineq} for the scalar balance law \eqref{e:balancelaw} in our selection principle.  Theorem \ref{def:solution} details the process by which our solution is constructed; we prove that the resulting object meets the requirements of Definition \ref{def:wksoln} below.  Finally, we explicitly connect our solution to the sticky particle dynamics \eqref{eq:CS}--\eqref{e:momentumconscollision}.  We demonstrate in Theorem \ref{t:SPconv} that the sticky particle Cucker--Smale dynamics can always be used to approximate the solution. Moreover, Theorem \ref{t:approxrate} gives a much stronger conclusion under additional hypotheses, by fashioning an explicit rate of convergence of the sticky particle approximation for the density profile.  

Let us denote by $\cP_c(\R)$ the space of probability measures with compact support. We will use the Wasserstein-1 metric to quantify the distance between elements of $\cP_c(\R)$.
\begin{DEF}[Wasserstein-1 metric]
  Let $\rho, \trho \in \cP_c(\R)$.
  The Wasserstein-1 distance between them is 
  \[\W_1(\rho, \trho) = \sup_{\Lip(f)\le 1} \left| \int_{\R} f(x)\dd\rho(x) - \int_{\R} f(x)\dd\trho(x) \right|.\]
\end{DEF}
It is well-known that if $M$ and $\tM$ are cumulative distribution
functions for $\rho$ and $\trho$ defined in \eqref{eq:MQ},
respectively, then $\W_1(\rho, \trho)= \|M-\tM\|_{L^1}$.
We will consider $\cP_c(\R)$ equipped with the Wasserstein-1 metric. In this setting, $\W_1$ convergence is equivalent to weak-$*$ convergence in the sense of measures.

Next, we make precise what we mean by a weak solution of
\eqref{eq:EA}.


\begin{DEF}[Weak solution]\label{def:wksoln}
  Let $\rho^0 \in \cP_c(\R)$ and $u^0 \in L^\infty(\dd\rho^0)$.
  Define $P^0=\rho^0u^0$, which lies in the space of signed measures $\cM(\R)$. We say that  $(\rho,P) = (\rho, \rho u)$ is a \emph{weak solution} to the
  Euler-alignment system \eqref{eq:EA} if for any $T>0$,
  \begin{itemize}
   \item $\rho\in C([0,T]; \cP_c (\R))$.
   \item $P(\cdot,t)\in \cM(\R)$ for any $t\in[0,T]$. Moreover,
    $P(\cdot,t)$ is absolutely continuous with respect to $\rho(\cdot,t)$,
    with the Radon-Nikodym derivative $u(\cdot,t)\in
    L^\infty(\dd\rho(t))$, where
    $u(\cdot,t)\dd\rho(\cdot,t)=\dd P(\cdot,t)$, for any $t\in[0,T]$.
    \item $(\rho,u)$ satisfies \eqref{eq:EA} in the sense of
      distributions.
    \item The initial condition $(\rho^0, P^0)$ is attained in the following weak
      sense for every $f\in C_c^\infty(\R)$:
      \begin{equation}
	\label{e:EAIC}
	\lim_{t\to 0+} \int_\R f(x) \dd\rho(x,t) = \int_\R f(x) \dd\rho^0(x);
	\quad 
	\lim_{t\to 0+} \int_\R f(x) \dd P(x,t) =  \int_\R f(x)\dd P^0(x).
      \end{equation}
  \end{itemize}
\end{DEF}

\subsection{Construction of the solution}
Let us start by introducing the generalized inverse of a nondecreasing function $M$, defined as
\[
M^{-1}(m) = \inf \{x\in\R : M(x)\ge m\},\qquad m\in(-\tfrac12, \tfrac12].
\]
It is a left-continuous function.

Now, we construct our solution through the procedure in the following theorem, which aligns with the formal derivation in Section \ref{ss:formalderiv}.
\begin{THEOREM}[The entropic selection principle]\label{def:solution}
  Let $\rho^0\in\cP_c(\R)$, $u^0\in L^\infty(\dd\rho^0)$ and
  $P^0=\rho^0u^0$.
  We construct a unique pair $(\rho, P)$ from the following procedure:
 \begin{itemize}
 \item[\textup{(i)}] Let $M^0(x) = \rho^0((-\infty, x])-\tfrac12$ and $\psi^0 = u^0 +
   \Phi*\rho^0$. Define a Lipschitz flux $A:[-\frac12,\frac12]\to\R$~by
   \begin{equation}\label{eq:A}
     A(m)=\int_{-\frac12}^ma(m')\dd m', \quad
     \text{where} \,\,\; a(m) = \psi^0 \circ (M^0)^{-1} (m).
   \end{equation}
 \item[\textup{(ii)}] Let $M$ be the unique entropy solution of \eqref{e:balancelaw} associated to the initial data $M^0$ and the flux $A$.  
 \item[\textup{(iii)}] Define $(\rho, P)$ from $M$ via the formulas
   \begin{equation}\label{eq:recover}
     \rho = \p_x M, \qquad P = -\p_t M = \p_x(A\circ M) - (\phi*M)\p_x M.
   \end{equation}
 \end{itemize}
 Then $(\rho, P)$ is a weak solution of the 1D Euler-alignment system
 \eqref{eq:EA} in the sense of Definition \ref{def:wksoln}.
 Moreover, we can define $u(\cdot, t)=\frac{\dd P(t)}{\dd\rho(t)}$ to be
 the Radon-Nikodym derivative of $P(t)$ with respect to $\rho(t)$.
$u(t)$ is uniquely defined $\rho(t)$-a.e. 
\end{THEOREM}

\begin{proof}
Our first step is to check that $M^0$ and $A$ satisfy the assumptions of Theorem \ref{t:M}. The required properties of $M^0$ follow directly from the fact that $\rho^0$ is a nonnegative, compactly supported probability measure.  Indeed, the number $R^0$ can be chosen so that $\supp\rho^0\subset[-R^0,R^0]$. As for $A$, we note that
\[
\|\psi^0\|_{L^\infty(\dd\rho^0)}\leq\|u^0\|_{L^\infty(\dd\rho^0)}+\Phi(2R^0),
\]
which is bounded.  It follows that $A$ as defined in \eqref{eq:A} is Lipschitz.  Since $M^0$ and $A$ are of the desired form, we can apply Theorem \ref{t:M} and obtain a unique entropy solution $M\in BV(\R\times[0,T])$ of \eqref{eq:M}, for any fixed time $T>0$.  

Now, we verify $(\rho, P)$ is a weak solution of \eqref{eq:EA}.
First, $\rho\in C([0,T]; \cP_c(\R))$ follows from
\eqref{eq:Mlimit}, and $P(t)=-\p_tM(t)\in \cM(\R)$ is a direct
consequence of \eqref{eq:Mtlimit}.
Let us turn our attention to $u$.
Since $M\in BV(\R\times[0,T])$, we can perform $BV$ calculus,
e.g. \cite[Lemma 4.2]{BouchutJames1999}, and deduce that
there exists a measurable function $\psi=\psi(x,t)$, bounded by $|A|_{\Lip}$ and uniquely defined $\rho(t)$-a.e., such that
\begin{equation}
  \label{e:gradAM}
  \p_x(A(M)) = \psi \p_x M, 
  \qquad 
  \p_t(A(M)) = \psi \p_t M,
\end{equation}
in the sense of measures. Then $P(t)$ defined by \eqref{eq:recover} is given by 
\[
P(t) = (\psi - (\phi*M))\p_x M(t) = (\psi - \Phi*\rho) \rho(t),
\]
so that $u = \psi - \Phi * \rho$ inherits the required boundedness and uniqueness properties from $\psi$ and $\rho$.

Next, we show that $(\rho, u)$ satisfies \eqref{eq:EA} in a
distributional sense.  From \eqref{eq:M}, \eqref{eq:recover}  and \eqref{e:gradAM} we get
\begin{align*}
& \p_t \rho 
= \p_t (\p_x M) = \p_x (\p_t M) = -\p_x P = -\p_x (\rho u),\\
& \p_t(\rho \psi) = \p_{tx}^2 (A(M)) = \p_x(\psi \p_tM)
= -\p_x( \psi P) = -\p_x (\rho u \psi).
\end{align*}
We then recover the momentum equation in exactly the form \eqref{eq:EA}$_2$ as follows:
\begin{align*}
\p_t (\rho u) + \p_x(\rho u^2) 
& = \big(\p_t (\rho \psi) + \p_x(\rho u\psi)\big) - \p_t\big(\rho (\Phi*\rho)\big) - \p_x\big(\rho u (\Phi*\rho)\big) \\
& = -  \big(\p_t \rho + \p_x(\rho u)\big)(\Phi*\rho) +\rho\big(\Phi*\p_x(\rho u)\big) - \rho u \big((\p_x \Phi)*\rho\big) \\
& = \rho (\phi*(\rho u)) - \rho u (\phi*\rho).
\end{align*}

Finally, we check the initial conditions \eqref{e:EAIC}.
From continuity of $\rho$ in time, we have $\W_1(\rho(t),\rho^0)\to0$,
which implies the first equation in \eqref{e:EAIC}.
For the second equation, we apply \eqref{eq:recover} to obtain
\[\int_\R f(x) \dd P(x,t) = -\int_\R f'(x) A(M(x,t))\dx - \int_\R f(x)(\phi*M)(x,t)\dd\rho(x,t).\]
We can pass to the limit as $t\to0+$ for the two terms separately.  For the first one, we have
\[
\left|\int_\R f'(x) A(M(x,t))\dx -\int_\R f'(x) A(M^0(x))\dx
\right|\leq |A|_{\Lip}\|f'\|_{L^\infty}\|M(\cdot,t)-M^0\|_{L^1}\to 0.
\]
As for the second one, we write
\begin{align*}
  &\left|\int_\R f(x)(\phi*M)(x,t)\dd\rho(x,t) -\int_\R f(x)(\phi*M^0)(x)\dd\rho^0(x)
  \right|\\
  &\leq \left|\int_\R f(x)\Big((\Phi*\rho)(x,t)-(\Phi*\rho^0)(x)\Big)\dd\rho(x,t) \right| + \left|\int_\R f(x)(\Phi*\rho^0)(x)\big[ \dd\rho(x,t)- \dd\rho^0(x) \big] 
    \right|.
\end{align*}
We note that the weak-$*$ convergence $\rho(t)\stackrel{*}{\rightharpoonup} \rho^0$ implies that the second term above vanishes as $t\to 0+$, and also that $\Phi*\rho(t)\to \Phi*\rho^0$ pointwise; the latter allows us to conclude that the first term also vanishes as $t\to 0+$, after an application of the dominated convergence theorem. 

From our construction, we have
$P(0) = (\psi^0-\Phi*\rho^0)\rho^0 = \rho^0 u^0 = P^0$, so the above calculations finish the proof.
\end{proof}

\vspace{- 3mm}
\subsection{Approximation by sticky particle Cucker--Smale dynamics}

\label{ss:approx}
One of the most important features of our entropic selection principle is that it associates to atomic initial data a solution of the Euler-alignment system that is described by the sticky particle Cucker--Smale dynamics. The proposition below gives the precise statement.  

\begin{PROP}
  Consider the 1D Euler-alignment system \eqref{eq:EA} with atomic initial data
  \begin{equation}\label{eq:rhoPN0}
    \rho_N^0(x) = \sum_{i=1}^{N} m_{i,N} \d(x-x_{i,N}^0),
    \qquad 
    P_N^0(x) = \sum_{i=1}^{N} m_{i,N} v_{i,N}^0 \d(x-x_{i,N}^0),
  \end{equation}
  where the $x_{i,N}^0$'s all belong to a fixed compact set and the $m_{i,N}$'s satisfy \eqref{e:mass1discrete}. Let $(x_{i,N}(t), v_{i,N}(t))$ be the solution of the sticky   particle Cucker--Smale dynamics with initial data $(m_{i,N}, x_{i,N}^0, v_{i,N}^0)_{i=1}^{N}$. The solution of the 1D Euler-alignment system selected by the procedure in Theorem \ref{def:solution} takes the form
  \begin{equation}\label{eq:rhoPNt}
    \rho_N(x,t) = \sum_{i=1}^{N} m_{i,N} \d(x-x_{i,N}(t)),
    \qquad 
    P_N(x,t) = \sum_{i=1}^{N} m_{i,N} v_{i,N}(t) \d(x-x_{i,N}(t)).
  \end{equation}
\end{PROP}

\begin{proof}
Let $M_N^0$, $\psi_N^0$, $A_N$, and $a_N$ be defined as in the step (i) of
the procedure in Theorem \ref{def:solution}. Let $\th_{i,N} =
\sum_{j=1}^i m_{j,N}$ for $i = 0,\ldots, N$.  Clearly $M_N^0$ can be
expressed as
\begin{equation}\label{eq:MNzero}
  M_N^0(x) = -\frac12 + \sum_{i=1}^N m_{i,N} H(x-x_{i,N}^0),
\end{equation}
and we have
\[
(M_N^0)^{-1}(m) = x_{i,N}^0, 
\qquad m\in (\th_{i-1},\th_i], \;i=1, \ldots, N.
\]
It follows that $a_N$ is piecewise constant and $A_N$ is piecewise linear, with breakpoints at the $\th_i$'s.  Thus the hypotheses of Theorem \ref{t:discr} are satisfied.  Furthermore, the quantities $\psi_{i,N}^0$ defined as in \eqref{eq:psiAn} satisfy 
\[
\psi_{i,N}^0:=A_N'(m) = a_N(m) = \psi(x_{i,N}^0), 
\qquad m\in (\th_{i-1},\th_i).
\]
We define $\widetilde{v}_{i,N}^0$ by \eqref{e:SPCSapprox0} and verify that it coincides with $v_{i,N}^0$ from \eqref{eq:rhoPN0}:
\[
\widetilde{v}_{i,N}^0 := \psi_{i,N}^0 - \sum_{j=1}^N m_{j,N} \Phi(x_{i,N}^0 - x_{j,N}^0) =  (\psi_N^0 - \Phi*\rho_N^0)(x_{i,N}^0) = v_{i,N}^0.
\]
Next, we apply Theorem \ref{t:discr} to obtain the entropy solution $M_N$ of \eqref{e:MN} associated to $M_N^0$ and $A_N$. According to \eqref{e:MNsoln} and \eqref{e:QN}, we have 
\begin{equation}
\label{e:MNQN}
M_N(x,t) = -\frac12 + \sum_{i=1}^{N} m_{i,N} H(x-x_{i,N}(t)),
\end{equation}
\begin{equation*}
A_N\circ M_N(x,t)= \sum_{i=1}^N m_{i,N}\psi_{i,N}(t)H(x-x_{i,N}(t)),
\end{equation*}

where 
\[
\psi_{i,N}(t) = v_{i,N}(t) + \sum_{j=1}^N m_{j,N} \Phi(x_{i,N}(t) - x_{j,N}(t)),\quad i=1, \ldots, N.
\]
It immediately follows that $\rho_N=\p_x M_N$ is given by \eqref{eq:rhoPNt}. 

As for $P_N$, we write
\begin{align*}
 P_N(x,t) &= \p_x(A_N\circ M_N)(x,t) - (\phi*M_N)\p_x M_N(x,t) \\
&= \sum_{i=1}^N m_{i,N} \psi_{i,N}(t) \d(x-x_i(t)) - (\Phi*\rho_N)(x,t) \sum_{i=1}^N m_{i,N} \d(x-x_{i,N}(t)) \\
& = \sum_{i=1}^N m_{i,N} \left[ \psi_{i,N}(t) - \sum_{j=1}^N m_{j,N}\Phi(x_{i,N}(t)-x_{j,N}(t)) \right] \d(x-x_{i,N}(t)) \\ & = \sum_{i=1}^N m_{i,N} v_{i,N}(t) \d(x-x_{i,N}(t)),
\end{align*}
which finishes the proof.
\end{proof}
A direct application of Theorem \ref{t:M}(b) gives a convergence result for the sticky particle approximation.

\begin{THEOREM}
\label{t:SPconv}
Let $(\rho^0, P^0)$ satisfy the hypotheses of Theorem
\ref{def:solution}, and let $(\rho, P)$ be the unique weak solution to
\eqref{eq:EA} that it generates.  There exists a sequence of
(explicitly constructed) atomic initial data $(\rho^0_N, P^0_N)_{N=1}^\infty$
such that if $(\rho_N, P_N)$ denotes the solution associated to
$(\rho_N^0, P_N^0)$ by the entropic selection principle, then for any
time $t>0$, as $N\to\infty$, we have
\begin{equation}\label{eq:W1conv}
   \W_1(\rho_N(t), \rho(t)) \to 0,
\end{equation}
and
\begin{equation}\label{eq:Pconv}
  P_N(t)\stackrel{*}{\rightharpoonup} P(t),\quad\text{in}\,\,\M(\R).
\end{equation}
\end{THEOREM}

\begin{proof}
Let $A$, $a$, $M$, and $(\rho, u)$ be defined from $(\rho^0, u^0)$ through Theorem \ref{def:solution}. We take $(\rho^0_N, P^0_N)$ as in \eqref{eq:rhoPN0}, with $(m_{i,N}, x_{i,N}^0)_{i=1}^N$ chosen according to \eqref{eq:miN} and \eqref{eq:xiN0}, and $(v_{i,N}^0)_{i=1}^N$ defined by
\begin{equation}\label{eq:v0}
v_{i,N}^0=\frac{1}{m_{i,N}}\int_{\th_{i-1,N}}^{\th_{i,N}}a(m)\,\dd
m-\sum_{j=1}^N m_{j,N}\Phi(x_{i,N}^0-x_{j,N}^0).
\end{equation}
By the proof of the above proposition, the solution $(\rho_N,P_N)$
generated by the entropic selection principle is given by
\eqref{eq:rhoPNt}. Moreover, defining discretized initial data $M_N^0$
and flux $A_N$ from $(\rho^0_N, P^0_N)$, we have that $M_N$ defined by
\eqref{e:MNQN} is the associated entropy solution of \eqref{e:MN}.
Now, we can apply Theorem \ref{t:M}(b). In particular,
\eqref{eq:W1conv} is equivalent to \eqref{eq:Mlimit}. Since 
$\p_t M_N = -P_N$ and $\p_t M = -P$ in the sense of measures, we get
\eqref{eq:Pconv} directly from \eqref{eq:Mtlimit}.
\end{proof}

Next, we provide a refined estimate of \eqref{eq:W1conv}, with an
explicit convergence rate on the sticky particle approximation to our
solution.

For $\rho^0\in \cP_c(\R)$, let us denote by $[x_\ell^0, x_r^0]$ the smallest interval such that $\supp\rho^0\subseteq [x_\ell^0, x_r^0]$. The diameter $D^0$ is defined as
\[D^0 = \diam\supp\rho^0:=x_r^0-x_\ell^0.\]

We start by constructing a well-prepared atomic approximation of
$\rho^0$, described in the next proposition.

\begin{PROP}\label{prop:rho0disc}
  Let $\rho^0\in\cP_c(\R)$. For any fixed $N\in \N$, there exists $(m_{i,N}, x^0_{i,N})_{i=1}^N$ such that 
\begin{equation}\label{eq:Minvdiff}
    \|M^0-M_N^0\|_{L^1(\R)}\leq \tfrac{D^0}{N},\qquad
    \|(M^0)^{-1} - (M_N^0)^{-1}\|_{L^\infty(-\frac12, \frac12]}\leq \tfrac{D^0}{N},
\end{equation}
where $M_N^0$ is defined in \eqref{eq:MNzero}.
\end{PROP}
\begin{proof}
  We first locate all the large internal vacuum intervals of $\rho^0$, namely
  $I_k=(a_k, b_k)\subset[x_\ell^0, x_r^0]$
  such that $b_k-a_k>\frac{D^0}{N}$ and $\rho^0(I_k)=0$.
  There are clearly fewer than $N$ such intervals.
  Let $S_N^{(1)}=\{a_k\}_{k=1}^K$ be the collection of all left endpoints of the $I_k$'s.  We take $N-K$ additional points, equally distributed, to partition $[x_\ell^0, x_r^0]$ into a total of $N$ intervals.  More precisely, we map $[x_\ell^0, x_r^0]\backslash\cup_{k=1}^K I_k$ to a single interval $[0, L]$, with $L<\frac{(N-K)D^0}{N}$, and we take equally distributed nodes $(\frac{i L}{N-K})_{i=1}^{N-L}$. Note that the distance between adjacent nodes is $\frac{L}{N-K}<\frac{D^0}{N}$. If the location of some node coincides with the image of a point in $S_N^{(1)}$, we can perturb the node slightly, in such a way that the distance in $[0,L]$ between any two adjacent nodes is still less than $\frac{D^0}{N}$. Now, we can take the inverse map of the selected nodes to $[x_\ell^0, x_r^0]\backslash\cup_{k=1}^K I_k$ and form a set $S_N^{(2)}$. We then order the set   $S_N=S_N^{(1)}\cup S_N^{(2)}$ to obtain an $N$-tuple  $(x^0_{i,N})_{i=1}^N$.
  Our construction is illustrated in Figure \ref{fig:x0}.
  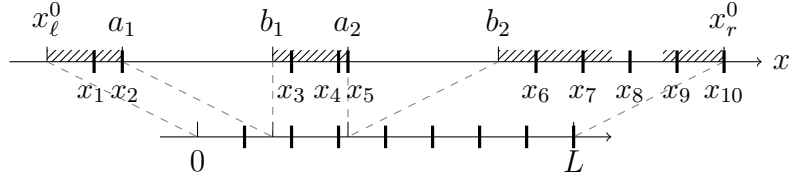
\begin{figure}[ht]
  \begin{tikzpicture}
    \draw[->] (-.5,0) -- (9.5,0) node[right] {$x$};
    \draw (0,0) -- (0,.2) node[above] {$x_\ell^0$};
    \draw (1,0) -- (1,.2) node[above] {$a_1$};
    \draw (3,0) -- (3,.2) node[above] {$b_1$};
    \draw (4,0) -- (4,.2) node[above] {$a_2$};
    \draw (6,0) -- (6,.2) node[above] {$b_2$};
    \draw (9,0) -- (9,.2) node[above] {$x_r^0$};
    \fill[pattern=north east lines] (0,0) rectangle (1,.15);
    \fill[pattern=north east lines] (3,0) rectangle (4,.15);
    \fill[pattern=north east lines] (6,0) rectangle (7.5,.15);
    \fill[pattern=north east lines] (8.2,0) rectangle (9,.15);

    \draw[->] (1.5,-1) -- (7.5,-1);
    \draw (2,-.8) -- (2,-1)  node[below,yshift=-1] {$0$};
    \draw (3, -.8) -- (3,-1);
    \draw (4, -.8) -- (4,-1);
    \draw (7, -.8) -- (7,-1) node[below,yshift=-1] {$L$};

    \draw[dashed,color=gray] (0,0) -- (2,-1);
    \draw[dashed,color=gray] (1,0) -- (3,-1) -- (3,0);
    \draw[dashed,color=gray] (4,0) -- (4,-1) -- (6,0);
    \draw[dashed,color=gray] (9,0) -- (7,-1);

    \foreach \i in {1,...,8}
       \draw[very thick] (2+\i*.625,-1.15) -- (2+\i*.625, -.85);

    \draw[very thick] (.625, .15) -- (.625, -.15)  node[below, xshift=-1] {{\small $x_1$}};
    \draw[very thick] (1,.15) -- (1, -.15)  node[below, xshift=1] {{\small $x_2$}};
    \draw[very thick] (2+2*.625,.15) -- (2+2*.625, -.15)  node[below] {{\small $x_3$}};
    \draw[very thick] (2+3*.625,.15) -- (2+3*.625, -.15)  node[below, xshift=-4] {{\small $x_4$}};
    \draw[very thick] (4,.15) -- (4, -.15)  node[below, xshift=5] {{\small $x_5$}};
      \foreach \i in {6,...,10}
       \draw[very thick] (2.75+\i*.625, .15) -- (2.75+\i*.625, -.15)
       node[below] {{\small $x_{\i}$}};
   \end{tikzpicture}
   \caption{ A sample construction of the $N$-tuple $(x_{i,N}^0)_{i=1}^N$. The shaded areas represent the support of $\rho^0$.  In this example, $N=10$, and there are two large vacuum intervals $(a_1,b_1)$ and $(a_2, b_2)$.  Thus $S_N^{(1)}=\{a_1, a_2\}$, and the set $S_N^{(2)}$ is constructed by taking 8 equally distributed points in $[0,L]$ and mapping back to $[x_\ell^0, x_r^0]$.}\label{fig:x0}  
 \end{figure}
 
We define an important quantity $x_{i,N}^*$ as follows. If $i-1\in S_1^{(N)}$, i.e., $x_{i-1,N}^0=a_k$, we set $x_{i,N}^*=b_k$. If $i-1\in S_2^{(N)}$, we set  $x_{i,N}^*=x_{i-1,N}^0$. For $i=1$, we set $x_{1,N}^*=x_\ell^0$. Our construction clearly guarantees
\begin{equation}\label{eq:xstar}
  0<x_{i,N}^0-x_{i,N}^*\le\tfrac{D^0}{N},\qquad i=1,\ldots,N.
\end{equation}

Next, we define $m_{i,N} = \rho^0((x_{i-1, N}^0, x_{i,N}^0])$ for each $i=1,\ldots,N$. (We take $x_{0,N}^0=-\infty$ for convenience.)
By construction, each $m_{i,N}$ is strictly positive.  As usual, define $\th_{i,N} = -\frac12 + \sum_{j=1}^i m_{j,N}$, $i=0, 1, \ldots, N$. It is easy to check that the $x_{i,N}^0$'s as defined above satisfy
\[
x_{i,N}^0 = \inf\{x:M^0(x)\ge \th_{i,N}\}, 
\qquad i=1, \ldots, N.
\]
Moreover, we have
\[M_N^0(x) = -\frac12 + \sum_{i=1}^N m_{i,N}H(x-x_{i,N}^0)\leq M^0(x),\]
with equality attained at $x=x_{i,N}^0$, as well as when $x\in
I_k$. This implies
\[\int_\R| M^0(x)-M^0_N(x)|\dx\leq \sum_{i=1}^N m_{i,N}
  (x_{i,N}^0-x_{i,N}^*)\leq \tfrac{D^0}{N}.\]
where we have used \eqref{eq:xstar} and $\sum_{i=1}^Nm_{i,N}=1$. 

For the second inequality in \eqref{eq:Minvdiff},
take any $i=1,\ldots,N$
and $m\in(\th_{i-1,N},\th_{i,N}]$.
By definition, $(M_N^0)^{-1}(m)=x_{i,N}^0$.
We claim that $(M^0)^{-1}(m)\in[x_{i,N}^*, x_{i,N}^0]$.
Indeed, in the case $x_{i-1,N}^0\in S_N^{(2)}$, we have
$M^0(x_{i,N}^*)=\th_{i-1,N}<m$ and therefore $(M^0)^{-1}(m)\geq x_{i,N}^*$; on the other hand, if  $x_{i-1,N}^0\in S_N^{(2)}$ or if $i=1$, then since $\rho\big(
(x_{i-1,N}^0,x_{i,N}^*)\big)=0$, we have
$M^0(x_{i,N}^*-)=\th_{i-1,N}<m$, which also implies
$(M^0)^{-1}(m)\geq x_{i,N}^*$.
Finally, we apply \eqref{eq:xstar} and conclude by writing
\[0\leq (M_N^0)^{-1}(m)-(M^0)^{-1}(m)\leq x_{i,N}^0-x_{i,N}^*\leq \tfrac{D^0}{N}.\]
\end{proof}

Under an additional regularity assumption on $\phi$ and on $u^0$, we can now apply the stability estimate \eqref{e:stability} in Theorem \ref{t:M}(c) to obtain an explicit error estimate for a sticky particle approximation to our solution, with the approximate initial density chosen according to the previous proposition.
\begin{THEOREM}
\label{t:approxrate}
Let $\b,s\in(0,1]$. Assume there exists a neighborhood of $0$ inside which $\phi(x)\le c_s|x|^{s-1}$ for some $c_s>0$.  Suppose $\rho^0\in\cP_c(\R)$ and $u^0\in C^\b([x_\ell^0, x_r^0])$.  There exists a sequence of (explicitly constructed) atomic initial data $(\rho^0_N, P^0_N)_{N=1}^\infty$ such that the corresponding solution of \eqref{eq:EA} satisfies
\begin{equation}\label{eq:W1decay}
\cW_1(\rho(t), \rho_N(t)) \le C(1+t)N^{-\g},
\qquad \g = \min\{s,\b\},
\end{equation}
for any $t>0$, where the constant $C$ depends on $D^0$, $\phi$, $c_s$, $s$, $\b$, and $|u^0|_{C^\b}$.
\end{THEOREM}
\begin{proof}
  First, we construct the approximated initial data $(\rho_N^0, P_N^0)$ via \eqref{eq:rhoPN0}, with $(m_{i,N}, x^0_{i,N})_{i=1}^N$ chosen according to Proposition \ref{prop:rho0disc}, and $(v_{i,N}^0)_{i=1}^N$ defined as follows.
  Set $\psi^0 = u^0 + \Phi*\rho^0$ as in Theorem~\ref{def:solution}. Note that our assumption on $\phi$ guarantees that $\Phi$ is locally $C^s$; the same is therefore true of $\Phi*\rho^0$. Thus
  $\psi^0\in C^\g([x_\ell^0, x_r^0])$, with
  \begin{equation}\label{eq:psiHolder}
    |\psi^0|_{C^\g([x_\ell^0, x_r^0])}\leq |u^0|_{C^\g([x_\ell^0, x_r^0])} +
    |\Phi*\rho^0|_{C^\g([x_\ell^0, x_r^0])}.
  \end{equation}
  We can then simply take $\psi_{i,N}^0=\psi^0(x^0_{i,N})$ and define $v_{i,N}^0$ using
  \eqref{e:SPCSapprox0}.

  Next, we apply the stability estimate \eqref{e:stability} with
  $\tM=M_N$, to get
  \[
    \cW_1(\rho(t), \rho_N(t))=\|M(\cdot,t)-M_N(\cdot,t)\|_{L^1(\R)} \le
    \|M^0-M_N^0\|_{L^1(\R)} + t\|a - a_N\|_{L^\infty(-\frac12, \frac12]}.
  \]
  The first term can be estimated directly from
  \eqref{eq:Minvdiff}. We calculate the second term as follows:
  \begin{align*}
    \|a - a_N\|_{L^\infty(-\frac12, \frac12]}=&\, \|\psi^0 \circ (M^0)^{-1} - \psi_N^0 \circ (M_N^0)^{-1} \|_{L^\infty(-\frac12, \frac12]} \\
\le&\, |\psi^0|_{C^\g([x_\ell^0, x_r^0])}\,\|(M^0)^{-1} - (M_N^0)^{-1}\|_{L^\infty(-\frac12, \frac12]}^\g + \|(\psi^0 - \psi_N^0)\circ (M_N^0)^{-1} \|_{L^\infty(-\frac12, \frac12]}\\
\le&\,  |\psi^0|_{C^\g([x_\ell^0, x_r^0])} N^{-\g}.
  \end{align*}
  Here, since $(M_N^0)^{-1}$ maps $(-\tfrac12,\tfrac12]$ to
  $(x_{i,N}^0)_{i=1}^N$, we only need to make sense of the function
  $\psi_N^0$ on $(x_{i,N}^0)_{i=1}^N$, where we have
  $\psi_N^0(x_{i,N}^0)=\psi^0(x_{i,N}^0)$ by construction.
  This eliminates the last term in the penultimate line.
  The last inequality is then obtained using \eqref{eq:Minvdiff}, and \eqref{eq:W1decay} follows immediately.
\end{proof}

\section{Asymptotic behavior of the solution}
\label{s:MoreProps}

In this section, we discuss the asymptotic behaviors of our weak
solutions $(\rho,P)$ to the 1D Euler-alignment system \eqref{eq:EA}.
The expected \emph{flocking phenomenon} has two ingredients.

First, we denote by $D(t)$ the diameter of the support of
$\rho(\cdot,t)$,
\begin{equation}\label{eq:Dt}
D(t):=\diam\supp \rho(t) = x_r(t)-x_\ell(t),
\end{equation}
where $[x_\ell(t), x_r(t)]$ is the smallest interval that contains
$\supp\rho(t)$.
We say that the solution experiences \emph{flocking} if
$D(t)$ remains uniformly bounded for all time, i.e., there exists a
constant $\overline{D}>0$ such that
\begin{equation}\label{eq:flocking}
  D(t)\leq \overline{D},\quad \text{for all}~~t\geq0.
\end{equation}

Second, we say that the solution experiences \emph{velocity alignment}
if the variation of the velocity $u(\cdot,t)$ decays to zero as time
approaches infinity.
Since our weak solutions $(\rho,P)$ only define the velocity
$u(\cdot,t)$ uniquely $\rho(t)$-a.e., we shall make sense of the
maximum and minimum velocities as follows:
\begin{equation}\label{eq:ubars}
u_+(t) = \sup_{f\in\cF}\frac{\int_\R f(x)\dd P(x,t)}{\int_\R
  f(x)\dd\rho(x,t)},\qquad
u_-(t) = \inf_{f\in\cF}\frac{\int_\R f(x)\dd P(x,t)}{\int_\R
  f(x)\dd\rho(x,t)},
\end{equation}
where
\[\cF =
  \left\{ f\in C_c^\infty(\R) : \int_\R f(x)\dd\rho(x,t)>0\right\}.\]
One can check that the definition of $u_+$ in
\eqref{eq:ubars} is equivalent to the \emph{essential supremum} of $u$:
\[u_+(t)=\inf\left\{c: \rho(t)\big(\{x: u(x,t)>c\}\big)=0\right\}.\]

Now, we are ready to define $V(t)$, the variation of $u(\cdot,t)$, by
\begin{equation}
\label{eq:Vt}
V(t) = u_+(t) - u_-(t).
\end{equation}
Velocity alignment happens when $V(t)\to0$ as time approaches
infinity. In particular, if $V(t)$ decays to zero exponentially in
time, we say the solution has \emph{fast alignment} property.

For regular solutions, it has been shown in \cite{TT2014} that if
$\phi$ has a fat tail \eqref{eq:fattail}, then
\emph{strong solutions must flock}:
any smooth solution $(\rho, u)$ of the Euler-alignment system
\eqref{e:EA} experiences flocking and fast alignment.
We will show that the same flocking phenomenon occurs for our weak
solutions.
Our strategy is to obtain flocking estimates for solutions to the sticky particle
Cucker--Smale dynamics which are uniform in the number of
particles $N$. Then we will use the convergence results for the approximations from Section \ref{ss:approx} to pass the properties to $(\rho,P)$.

\subsection{Uniform flocking estimates on the sticky particle approximations}
The flocking phenomenon for the Cucker--Smale dynamics \eqref{e:CS} has
been first studied for general protocols $\phi$ in \cite{HaLiu2008}. The idea can be easily adapted
to the sticky particle dynamics.

Consider a sequence of sticky particle Cucker--Smale dynamics
$(m_{i,N},x_{i,N}(t), v_{i,N}(t))$ that approximates the
Euler-alignment system $(\rho, P)$.
The discrete analog of the diameter is
\begin{equation}
\label{e:DNdef}
D_N(t) := \max_{1\le i,j \le N} |x_{i,N}(t) - x_{j,N}(t)| = x_{N,N}(t)-x_{1,N}(t),
\end{equation}
and the variation of velocity becomes
\begin{equation}
\label{e:VNdef}
V_N(t) := \max_{1\le i,j \le N} |v_{i,N}(t) - v_{j,N}(t)| = \max_{1\le i \le N}v_{i,N}(t)-\min_{1\le i \le N}v_{i,N}(t).
\end{equation}
Note that using the approximations of the initial data
$(m_{i,N},x_{i,N}^0, v_{i,N}^0)$ constructed in Theorem \ref{t:SPconv}, it is easy
to verify that
\begin{equation}\label{eq:DV0}
  D_N(0)\leq D^0,\quad V_N(0)\leq V^0.
\end{equation}

We are ready to establish uniform flocking estimates for the sticky particle Cucker--Smale dynamics. 
\begin{THEOREM}
  \label{t:SPflocking}
  Let $(x_{i,N}(t), v_{i,N}(t))_{i=1}^N$ be a sequence of
  sticky particle Cucker--Smale dynamics associated to the initial
  data $(m_{i,N}, x_{i,N}^0, v_{i,N}^0)_{i=1}^N$.  Define $D_N(t)$ and
  $V_N(t)$ as in \eqref{e:DNdef} and \eqref{e:VNdef}. Assume
  \eqref{eq:DV0} holds, and that
  \begin{equation}
    \label{e:flockingthreshold}
    \sup_{R>0} \Phi(R) > \cE^0:=\Phi(D^0)+V^0.
  \end{equation}
  Then for all $t\ge 0$, the sticky particle dynamics satisfy the following estimates, uniformly in $N$.
  \begin{align}
    \label{e:SPflocking}
    \text{Flocking}: & \qquad \sup_{t\ge 0} D_N(t) \le \overline{D}:=\Phi^{-1}(\cE^0)<+\infty; \\
    \label{e:SPalignment} 
    \text{Fast alignment}: & \qquad V_N(t) \le V^0\exp(-\phi(\overline{D})t);
  \end{align}
\end{THEOREM}

\begin{proof}	
  It is well-known that on time intervals where $(x_i(t), v_i(t))_{i=1}^N$ follow the (collisionless) Cucker--Smale dynamics, the quantities $D_N(t)$ and $V_N(t)$ are Lipschitz continuous and satisfy the following differential inequalities at every time $t$ where they are differentiable:
  \begin{align}
    \label{e:SDDID}
    \dot{D}_N(t) & \le V_N(t),\\
    \label{e:SDDIV}
    \dot{V}_N(t) & \le -\phi(D_N(t))V_N(t).
  \end{align}
Define a Lyapunov functional $\cE_N(t)=\Phi(D_N(t))+V_N(t)$.
It clearly follows from \eqref{e:SDDID} and \eqref{e:SDDIV} that $\cE_N(t)$ is nonincreasing along intervals during which no collisions occur.  On the other hand, if $t$ is a collision time, then $V_N(t)\le V_N(t-)$ by the maximum principle, while $\Phi(D_N(t))$ is continuous at time $t$.  Therefore $\cE_N(t)$ is nonincreasing on all of $[0,\infty)$.  
	
Let us now assume that \eqref{e:flockingthreshold} holds.  Since
$\cE_N(t)$ is nonincreasing, we have
\[\Phi(D_N(t))\le \cE_N(t)\le \cE_N(0) \le \cE^0,\quad \text{for
    all}\,\, t\ge 0.\]
This implies \eqref{e:SPflocking}.
Note that $\Phi$ is a nondecreasing function in $[0,\infty)$, and the assumption
  \eqref{e:flockingthreshold} guarantees that $\cE^0$ lies in its
  range. Therefore, $\overline{D}=\Phi^{-1}(\cE^0)=\inf\{R:
  \Phi(R)\geq \cE^0\}$ is well-defined and takes a
  nonnegative finite value.

Finally, combining \eqref{e:SPflocking} and
\eqref{e:SDDIV} and using the fact that $\phi$ is radially decreasing,
we obtain \eqref{e:SPalignment}.
\end{proof}

\subsection{Flocking for the 1D Euler-alignment system}

\begin{THEOREM}\label{t:EAflocking}
Let $(\rho^0, P^0)$ satisfy the hypotheses of Theorem
\ref{def:solution}, and $(\rho, P)$ be the associated solution.
Assume \eqref{e:flockingthreshold} holds.
Then the solution $(\rho, P)$ experiences
 \begin{align}
    \label{e:EAflocking}
    \text{Flocking}: & \qquad \sup_{t\ge 0} D(t) \le \overline{D}:=\Phi^{-1}(\cE^0)<+\infty; \\
    \label{e:EAalignment} 
    \text{Fast alignment}: & \qquad V(t) \le V^0\exp(-\phi(\overline{D})t); 
  \end{align}
\end{THEOREM}

\begin{proof}
  Let $(\rho_N, P_N)_{N=1}^\infty$ be a sequence of sticky particle
  approximations. We will establish \eqref{e:EAflocking} and \eqref{e:EAalignment} using the uniform flocking estimates on $(\rho_N, P_N)$ furnished by Theorem \ref{t:SPflocking}, as well as the convergence results for the sticky particle approximation in Theorem \ref{t:SPconv}.

  We first show \eqref{e:EAflocking} by contradiction. Assume there
  exists a time $t$ such that $D(t) = \overline{D}+\e$, with some
  $\e>0$. We apply Lemma \ref{lem:flockingconv} below with
  $\rho=\rho(t)$ and $\trho=\rho_N(t)$ and get
  $\W_1(\rho(t), \rho_N(t))\geq c>0$, where $c=c(\rho,\e)$ is
  independent of $N$. This uniform positive lower bound contradicts the convergence~\eqref{eq:W1conv}. 

  Next, we turn to \eqref{e:EAalignment}. Fix a time $t\geq0$.
  In view of \eqref{eq:ubars}, we can find a sequence of test
  functions $f_k\in C_c^\infty(\R)$, 
 normalized by $\int_\R f_k(x)\dd\rho(x,t)=1$, such
 that $u_+(t)=\lim_{k\to\infty}\int_\R f_k(x)\dd P(x,t)$.
 For each $k$, we apply the convergence results \eqref{eq:W1conv} and
 \eqref{eq:Pconv} to get
 \[
 \lim_{N\to\infty}\int_\R f_k(x)\dd\rho_N(x,t)=\int_\R f_k(x)\dd\rho(x,t) =1,\quad
  \lim_{N\to \infty} \int_\R f_k(x)\dd P_N(x,t)=\int_\R f_k\dd P(x,t).
  \]
 We therefore obtain
 \[
   \int_\R f_k(x) \dd P(x,t) = \lim_{N\to \infty}
   \frac{\int_\R f_k(x)\dd P_N(x,t)}{\int_\R f_k(x) \dd\rho_N(x,t)}.
\]
Similarly, we find a sequence of normalized test functions $g_k$ such that $u_-(t)=\lim_{k\to\infty}\int_\R g_k(x)\dd
P(x,t)$. We may thus write
 \begin{align*}
  &\int_\R f_k(x) \dd P(x,t) -\int_\R g_k(x) \dd P(x,t) = \lim_{N\to \infty}
   \left(\frac{\int_\R f_k(x)\dd P_N(x,t)}{\int_\R f_k(x) \dd\rho_N(x,t)}
  -\frac{\int_\R g_k(x)\dd P_N(x,t)}{\int_\R g_k(x) \dd\rho_N(x,t)}\right)\\
 &\leq \limsup_{N\to\infty} \left(\max_{1\le i\le N}v_{i,N}(t)-\min_{1\le i\le N}v_{i,N}(t)\right)
   =\limsup_{N\to\infty} V_N(t)\leq V^0\exp(-\phi(\overline{D})t),
 \end{align*}
where we have used the uniform fast alignment estimate
\eqref{e:SPalignment} in the last inequality.
Finally, we take $k\to\infty$ and use the definition \eqref{eq:Vt} to
conclude \eqref{e:EAalignment}.
\end{proof}

When the communication protocol $\phi$ has a fat tail \eqref{eq:fattail},
we get $\lim_{R\to\infty}\Phi(R)=\infty$. Therefore,
\eqref{e:flockingthreshold} holds for any finite $D^0$ and $V^0$.
Hence, Theorem \ref{t:EAflocking} implies that our \emph{weak solution
must flock}.

We now present the following lemma, which is used in the proof
of Theorem \ref{t:EAflocking}. 
\begin{LEMMA}\label{lem:flockingconv}
  Let $\rho, \trho\in\cP_c(\R)$. Denote $D$ and $\tD$ the diameter of
  the support of $\rho$ and $\trho$, respectively. Suppose $\tD<D$.
  Then, there exists a constant $c>0$, depending only on $\rho$ and $D-\tD$, such that $\W_1(\rho,\trho)\geq c$.
\end{LEMMA}
\begin{proof}
  Let $[x_\ell, x_r]$ be the smallest interval that contains $\supp\rho$, so that $D=x_r-x_\ell$. Define $\e=D-\tD>0$.
  Since $\supp \trho$ has the smaller diameter, we must have $\trho(I)=0$ for
  at least one of the intervals  $I=[x_\ell, x_\ell+\tfrac\e2)$ or
  $(x_r-\tfrac\e2,x_r]$. 
 Consider the first case. Let $f_N$ be a Lipschitz function which is supported in $[x_\ell-\e, x_\ell+\tfrac\e2)$, takes the value 1
  in $[x_\ell(t), x_\ell(t)+\tfrac\e4]$, and satisfies $|f_N|_{\Lip}\leq 4\e^{-1}$. Then
  \[\W_1(\rho, \trho)\geq
    \frac{4}{\e}\int_{x_\ell}^{x_\ell+\tfrac\e4}\dd\rho(x) =
    c_l(\rho,\e)>0.\]
  A similar argument works for the other case and yields
  $\W_1(\rho, \trho)\geq c_r(\rho, \e)>0$. 
  Thus, we get a positive lower bound $c=\min\{c_l(\rho,\e),
  c_r(\rho,\e)\}>0$, 
  which only depends on $\rho$ and $\e$.
\end{proof}

\subsection{Remarks on strong flocking}

Suppose that the communication protocol has a fat tail \eqref{eq:fattail}, or more generally that \eqref{e:flockingthreshold} holds. For classical solutions, it is known that in addition to flocking and fast alignment, one has convergence of the density profile to a traveling wave:
\begin{equation}
\label{e:strongflocking}
\cW_1(\rho(\cdot + \overline{u} t, t), \rho_\infty)\to 0,  \qquad \text{ as } t\to +\infty.
\end{equation}
Here $\overline{u} = \frac{\int \rho^0 u^0\dx}{\int \rho^0 \dx}$ denotes the average velocity and $\rho_\infty \in \cP_c(\R)$ is some compactly supported probability measure, which can be obtained by pushing forward $\rho^0$ under the flow map.  A convergence like \eqref{e:strongflocking} is sometimes referred to as \textit{strong} flocking (c.f. \cite{ShvydkoyBook}). We claim that strong flocking occurs in our entropically selected 1D weak solutions as well, under the assumption \eqref{e:flockingthreshold}.  

We argue by approximation.  Let $(\rho, P)$ be a weak solution with initial data $(\rho^0, P^0=\rho^0 u^0)$.  Let $(\rho_N, P_N)$ be a sequence of sticky particle solutions constructed as in Theorem \ref{t:SPconv} and its proof.  Note that \eqref{eq:v0} guarantees that $\sum_{i=1}^N m_{i,N} v_{i,N}^0 = \int \rho^0 u^0\dx$ for each $N$; therefore we may assume without loss of generality that both are zero, by Galilean invariance.

Assume \eqref{e:flockingthreshold} holds for $(\rho^0, u^0)$.  We have argued above that \eqref{eq:DV0}--\eqref{e:flockingthreshold} then hold for $(\rho_N^0, u_N^0)$ as a consequence.  Therefore, 
\begin{align*}
\cW_1(\rho_N(t_1), \rho_N(t_2)) 
& \le \sum_{i=1}^N m_{i,N} |x_{i,N}(t_1) - x_{i,N}(t_2)| \le \sum_{i=1}^N m_{i,N} \int_{t_1}^{t_2} |v_{i,N}(t)|\dt \\
& \le \sum_{i=1}^N m_{i,N} \int_{t_1}^{t_2} 2V^0 \exp(-\phi(\overline{D}t))\dt \\
& \le \frac{2V^0}{\phi(\overline{D})} \exp(-\phi(\overline{D})t_1), 
\end{align*}
for $0\le t_1 \le t_2$.  Pick $t_1$ large enough so that the right side of the above is small.  Choosing any $t_2\ge t_1$, then $N=N(t_1, t_2)$ large enough, we may conclude that 
\[
\cW_1(\rho(t_1), \rho(t_2)) \le \cW_1(\rho(t_1), \rho_N(t_1)) + \cW_1(\rho_N(t_1), \rho_N(t_2)) + \cW_1(\rho_N(t_2), \rho(t_2))
\]
can be made as small as desired.  Since $(\rho(t))_{t>0}$ is thus Cauchy in $(\cP_c(\R), \cW_1)$, we may conclude the existence of a $\rho_\infty$ such that \eqref{e:strongflocking} holds.  

The structure of the measure $\rho_\infty$ has been analyzed extensively in \cite{LLST2020geometric} in the context of classical solutions (c.f. also \cite{LS2019}).  In \cite{LLST2020geometric}, the authors studied \textit{mass concentration} in the limiting profile $\rho_\infty$ by demonstrating a correspondence between the set where $e^0$ is zero and the singular support of $\rho_\infty$. It would be interesting to study the structure of $\rho_\infty$ for the weak solutions considered in this paper---it is far from straightforward to predict the outcome from the initial data.  We leave further consideration of this question for future work.

\medskip
{\bf Acknowledgments.}
This material is based upon work supported by the National
Science Foundation under Grant No. DMS-1928930 while TL participated
in a program hosted by the Mathematical Sciences Research Institute
in Berkeley, California, during the Spring 2021 semester.

CT acknowledges the support of NSF grants DMS-1853001 and DMS-2108264.

Both authors thank Roman Shvydkoy and Eitan Tadmor for insightful comments on an earlier draft. We also thank the anonymous referees for their valuable feedback, which pushed us to strengthen our original results.


\def\cprime{$'$}

\end{document}